\documentclass[12pt,draftclsnofoot,onecolumn]{IEEEtran}
\usepackage{amsmath,graphicx}
\usepackage{algorithm}
\usepackage{algorithmic}
\usepackage{enumerate}
\usepackage{amssymb}
\usepackage{amsmath}
\usepackage{amsthm}
\usepackage{epstopdf}
\usepackage{subfig}
\usepackage{makecell}
\usepackage{multirow}
\usepackage{pifont,color}
\usepackage[dvipsnames]{xcolor}

\def\MatrixFont{\bf}
\def\VectorFont{\bf}

\newcommand{\mG}{{\MatrixFont G}}
\newcommand{\mH}{{\MatrixFont H}}
\newcommand{\mI}{{\MatrixFont I}}

\newcommand{\mN}{{\MatrixFont N}}

\newcommand{\mQ}{{\MatrixFont Q}}
\newcommand{\mR}{{\MatrixFont R}}
\newcommand{\mS}{{\MatrixFont S}}

\newcommand{\mV}{{\MatrixFont V}}
\newcommand{\mW}{{\MatrixFont W}}
\newcommand{\mX}{{\MatrixFont X}}
\newcommand{\mY}{{\MatrixFont Y}}
\newcommand{\mZ}{{\MatrixFont Z}}

\newcommand{\ve}{{\VectorFont e}}

\newcommand{\vh}{{\VectorFont h}}

\newcommand{\vn}{{\VectorFont n}}

\newcommand{\vq}{{\VectorFont q}}

\newcommand{\vs}{{\VectorFont s}}

\newcommand{\vu}{{\VectorFont u}}
\newcommand{\vv}{{\VectorFont v}}
\newcommand{\vw}{{\VectorFont w}}
\newcommand{\vx}{{\VectorFont x}}
\newcommand{\vy}{{\VectorFont y}}

\newtheorem{theorem}{Theorem}
\newtheorem{lemma}{Lemma}

\newtheorem{proposition}{Proposition}

\newtheorem{definition}{Definition}

\newcommand\tr{\ensuremath{{\rm Tr}}}

\ifCLASSINFOpdf
\else
\fi
\begin{document}
%
\title{Efficient Global Algorithms for Transmit Beamforming Design in ISAC Systems}
%
%

\author{Jiageng Wu, Zhiguo Wang, Ya-Feng Liu, and Fan Liu
\thanks{J. Wu is with  the School of Mathematics, Jilin University, Changchun ~130012, China (e-mail: wujg22@mails.jlu.edu.cn). Z. Wang is with  the College of Mathematics, Sichuan
University, Chengdu, Sichuan 610064, China (e-mail: wangzhiguo@scu.edu.cn).  Y.-F. Liu is with the State Key Laboratory of Scientific and Engineering Computing, Institute of Computational Mathematics and Scientific/Engineering
Computing, Academy of Mathematics and Systems Science, Chinese Academy
of Sciences, Beijing 100190, China (e-mail: yafliu@lsec.cc.ac.cn). F. Liu is with the Department of Electrical and Electronic Engineering, Southern University of Science and Technology, Shenzhen 518055, China (e-mail: liuf6@sustech.edu.cn).  (Corresponding
author: Zhiguo Wang.) }
}

\maketitle

\begin{abstract}
In this paper,  we propose a multi-input multi-output transmit beamforming  optimization model for joint radar sensing and multi-user communications, where the design of the beamformers is formulated as an optimization problem whose objective is a weighted combination of the sum rate and the Cram\'{e}r-Rao bound, subject to the transmit power budget. Obtaining the global solution for the formulated nonconvex problem is a challenging task, since the sum-rate maximization problem itself   (even without considering the sensing metric) is known to be NP-hard. The main contributions of this paper are  threefold. Firstly, we derive an optimal closed-form   solution to the formulated problem in the single-user case and the multi-user case where the channel vectors of different users are orthogonal. Secondly, for the general multi-user case, we propose a novel branch and bound  (B\&B) algorithm  based on the McCormick envelope relaxation. The proposed algorithm is guaranteed to find the globally optimal solution to the formulated problem.  Thirdly, we  design a  graph neural network (GNN) based pruning policy to determine
irrelevant nodes that can be directly pruned in the proposed B\&B algorithm, thereby significantly reducing the number of
unnecessary enumerations  therein and improving its computational efficiency.  Simulation results show the efficiency of the proposed vanilla and GNN-based accelerated B\&B algorithms.
\end{abstract}


%
\IEEEpeerreviewmaketitle

\vspace{-0.3cm}
\section{Introduction}
\vspace{-0.0cm}
\label{sec:intro}
Future wireless networks are anticipated to deliver enhanced communication services and support a range of emerging applications  related to both sensing and communications  such as smart manufacturing, environment monitoring, and remote
health-care \cite{tan2018exploiting,ma2020joint}.   Integrated sensing and communication (ISAC), in which  radar sensing and wireless communications are integrated to share a common radio spectrum, has gained significant attention and recognition from both academia and industry.  ISAC has been envisioned as a pivotal enabler to support various sensing and communication applications \cite{su2022secure,xiong2023fundamental,song2023intelligent}. In particular, ISAC has been listed as one of the six key usage scenarios of the future 6G system  in the new
Recommendation \cite{ITU-R-WP5D2023} for IMT-2030 (6G).

Unlike traditional systems, where communication and radar sensing are typically designed and deployed independently, ISAC aims to achieve dual  functionalities within a unified system through shared frequency, hardware, and joint signal processing design \cite{zhang2021overview}. There are three categories of the design methodology: radar-centric design \cite{mealey1963method}, communication-centric design \cite{zou2023energy}, and joint design \cite{hassanien2015dual}. All of these  design methodologies have shown that ISAC significantly enhances spectral efficiency and reduces implementation costs by sharing spectral resources and reusing expensive hardware platforms. In particular, compared  against radar-centric and communication-centric designs, the joint design methodology offers greater flexibility in the beamforming/waveform design, potentially providing a much better trade-off between communication and sensing.
Therefore, transmit beamforming designs in the category of joint design have gained growing interests recently \cite{liu2022integrated,liu2020joint}. This motivates us to focus on the joint beamforming design.

There are two formulations for multi-user beamforming designs for ISAC systems from different perspectives. From the perspective of users \cite{liu2021cramer}, the beamforming design problem is formulated as the
 Cram\'{e}r-Rao bound (CRB)-minimization problem  under  individual signal-to-interference-plus-noise ratio (SINR) constraints of all communication users \cite{liu2022joint}. These  SINR constraints guarantee a minimum level of communication quality of service  (QoS) for all  communication users.  It has been shown in  \cite{liu2022joint} that the proposed beamforming design  based on the CRB metric significantly outperforms  the corresponding counterparts based on beampattern matching \cite{liu2020joint} in terms of the sensing performance.
From the perspective of the system operator \cite{weeraddana2012weighted}, the beamforming design problem is often formulated as the sum-rate maximization problem under the total power
budget constraint at the transmitter \cite{shi2011iteratively,shen2018fractional}. It is worth noting that the sum rate is a more fundamental metric that characterizes the overall performance in multi-user communication scenarios. Motivated by this, this paper is interested in optimizing both communication performance, measured by the  sum rate of all communication users, and the target estimation performance, measured by the CRB for unbiased estimators.  In contrast to individual SINR constraints  of all communication users in \cite{liu2020joint}, our design incorporates sum-rate maximization (SRM) to ensure a superior network throughput performance.

Specifically, this paper considers the transmit beamforming design problem of optimizing  a
weighted combination of the sum rate of all communication
users and the CRB of the sensing target under the total transmit
power constraint. The paper proposes an efficient  global branch and bound (B\&B)
algorithm for solving the formulated
nonconvex problem to characterize the trade-off  between two fundamental metrics of communication and sensing, which are sum rate and CRB. The computation of   the global solution holds significant importance in the sense that the resultant
global solution serves as vital benchmarks for assessing the
performance of existing heuristic or local algorithms designed
for the same problem. Finally, this paper further proposes an accelerated B\&B algorithm
based on a  graph neural network (GNN) model inspired by the work \cite{shrestha2023optimal} to improve the computational efficiency of the proposed vanilla B\&B algorithm.
\vspace{-1mm}
\subsection{Related Works}
\textbf{Global solution for SRM.}  SRM plays a critical
role in wireless communication system design \cite{liu2024survey}. However,  the SRM
 problem is more difficult compared to the  counterpart in the fixed SINR target case. Notably, the SRM problem has been proven to be NP-hard \cite{chiang2007power,luo2008dynamic}. There are generally two popular methods for achieving the global solution to the SRM problem \cite{weeraddana2011weighted,matthiesen2020mixed}: the outer polyblock approximation (PA) algorithm \cite{qian2009mapel} and the B\&B algorithm \cite{tuy2005monotonic}. In general, the success of the PA algorithm depends on the property that the objective function of SRM is monotonically increasing in their SINRs \cite{liu2012achieving}. However, this monotonicity property does not hold for our interested problem as its objective function also includes a CRB term for radar sensing. Alternatively, the B\&B algorithm  uses a tree search strategy to implicitly enumerate all of its possible solutions \cite{tuy2005monotonic}.  The numerical efficiency and convergence speed of the B\&B algorithm heavily rely on the quality of upper and lower bounds for the optimal value \cite{liu2024survey,weeraddana2011weighted,lu2017efficient,lu2020enhanced,schobel2010theoretical}.

\textbf{Local solution for ISAC with SRM.} State-of-the-art iterative  local algorithms, such as weighted sum-minimum mean-square error (WMMSE) \cite{shi2011iteratively}, fractional programming (FP) \cite{shen2018fractional}, and minorization-maximization (MM) \cite{sun2016majorization}, primarily aim to find a stationary point for SRM. In ISAC scenarios involving both SRM and radar sensing, the work \cite{zhu2023information} combined WMMSE and  semidefinite relaxation (SDR) techniques, named WMMSE-SDR, to attain a Karush-Kuhn-Tucker (KKT) point. Additionally, the work \cite{luo2022joint} proposed an efficient alternating  optimization algorithm based on FP to jointly optimize the performance of the ISAC system. In particular, the above algorithm seeks to maximize the achievable sum rate of communication users while satisfying the beampattern similarity constraint for radar sensing. When accounting for the impact of spatially correlated channels in ISAC, the work \cite{wang2022noma} introduced a successive convex approximation algorithm for solving the weighted combination of SRM and the effective sensing power. It is  worth mentioning that all of  aforementioned algorithms  cannot be guaranteed to find the globally optimal solution for the  problem under consideration.

\textbf{Machine learning based accelerated B\&B algorithms.}
The B\&B algorithm, essentially   a ``smart"  enumerative approach to  globally  solve nonconvex optimization problems, exhibits inherent weaknesses in  its scalability
to solve large-scale problems. To  overcome this limitation, the work \cite{he2014learning} pioneered the integration of a machine learning (ML) model that involves  an offline training of a binary classifier utilizing diverse problem instances into the B\&B algorithm. Subsequently, the trained classifier predicts and bypasses nodes that do not contain optimal solutions, which are named as irrelevant codes during the B\&B algorithm's branching process, thereby resulting in substantial time savings. Building upon this foundation, the work \cite{shen2019lorm} introduced a pruning policy based on a multi-layer perceptron (MLP) specifically tailored for wireless networks. Notably, the recent work \cite{shrestha2023optimal} demonstrated that a GNN-based model can significantly accelerate the B\&B algorithm at the same time preserves its global optimality with high probability.

%
%
%
%
%

\vspace{-2mm}
\subsection{Our Contributions}
This paper proposes the first tailored efficient global algorithm for solving the joint SRM and CRB-minimization problem, which characterizes a trade-off between two fundamental metrics of communication and sensing. The proposed approach is in sharp contrast to existing works, which either primarily concentrate on the SRM problem or resort to approximation (or local) optimization algorithms for solving transmit beamforming design problems in ISAC systems. The main contributions of this paper are summarized as follows:
\begin{itemize}
  \item  We formulate a joint optimization problem, whose objective is a weighted combination of the sum rate of all communication users and the CRB of the  extended target. The formulated problem strikes an important trade-off between two fundamental performance metrics of communication and sensing in ISAC systems. There are significant technical challenges in globally solving the formulated problem, as the considered problem without the CRB term reduces to the NP-hard SRM problem.
  \item We derive the optimal closed-form  solution  to the formulated problem in two special cases, which are the single-user case and the  multi-user case where the channel vectors of different users are orthogonal. In the latter case,
we prove that the  all multi-user interference terms are minimized to be zero  at the optimal solution, which further implies that the original nonconvex problem can be transformed into an equivalent convex problem.
  \item In the general multi-user case, we propose a novel B\&B algorithm that leverages the McCormick envelope relaxation. We show that the proposed B\&B algorithm is guaranteed to find the globally optimal  solution   of the formulated problem. Global optimization algorithms are vital in accessing the fundamental trade-off between communication and sensing in ISAC systems,
      serving as benchmarks for  evaluating the performance of existing
heuristic or local algorithms designed for the same problem.
  \item In order to improve its computational efficiency, we  design a  GNN-based binary classifier in the proposed B\&B algorithm. The classifier  is helpful in determining irrelevant nodes that can be directly pruned in the B\&B algorithm,
      thereby significantly reducing the number of unnecessary enumerations and accelerating the B\&B process.
\end{itemize}
Simulation results verify the derived theoretical results  including  closed-form solutions in two special cases and global optimality of the proposed B\&B algorithm, and show that the GNN-based accelerated B\&B algorithm achieves an order-of-magnitude speedup compared to the  vanilla B\&B algorithm.

The prior work \cite{wang2023globally} solely focused on the design of the B\&B algorithm for the general multi-user case. The present paper is a significant extension of \cite{wang2023globally}.  More specifically, we provide theoretical analysis for both the single-user case and the orthogonal multi-user case and propose an GNN-based accelerated B\&B algorithm. These results are completely new compared with our prior work. Finally, we conduct extensive simulation results to verify  all of derived theoretical results and compare our proposed algorithm with state-of-the-art benchmarks.
\vspace{-2mm}
\subsection{Organization and Notations}
The organization of this paper is as follows. In Section \ref{sec_system}, the joint
 beamforming design problem  is formulated. In Section \ref{sec_property},
  a few important properties of the optimal solution for the joint beamforming design problem are derived. Based on these properties, a global B\&B algorithm for solving the problem is proposed in Section \ref{sec_BB}. Section \ref{sec_ml_bb} proposes an GNN-based accelerated B\&B algorithm.
Simulation results are presented in Section \ref{sec_num} and the paper is finally concluded
in Section \ref{sec_con}.

\emph{Notations:} Vectors and matrices are represented by bold lowercase letters (i.e., $\vw$) and bold uppercase letters (i.e., $\mW$), respectively.  $\tr(\cdot)$, $(\cdot)^\top$,  $(\cdot)^H$, and $(\cdot)^{-1}$ denote the trace operator, the transpose operator,  the Hermitian transpose operator, and the inverse operator, respectively. We use  $\mI_K$ to denote the
$K\times K$ identity matrix, $\|\vw\|$ to denote the Euclidean norm of vector $\vw$, and $[K]$ to denote the set $\{1,2,\ldots,K\}$. $\mX\succ 0$ ($\mX\succeq 0$) denotes that $\mX$ is a positive (semi)definite matrix. Finally, we use $\textmd{Re}(\cdot)$ and $\textmd{Im}(\cdot)$ to denote
the  real and imaginary parts of a complex value, respectively.
\section{System Model and Problem Formulation}\label{sec_system}
\subsection{System Model}
Consider a   multi-input multi-output (MIMO) ISAC base station (BS) equipped with $N_t$ transmit
antennas and $N_r$ receive antennas, which serves $K$
downlink single-antenna users while detecting an extended target.  Without loss of generality, we assume  $N_t<N_r$ in this paper in order to avoid the information loss of the sensed target.

Let $\mX\in \mathbb{C}^{N_t\times L}$ be the transmitted baseband signal  with $L>N_t$ being the length of the radar pulse/communication
frame. The matrix $\mX$ is  the sum of linear precoded radar waveforms and communication symbols, given by
\vspace{-1mm}
\begin{align}\label{eqn:X}
\mX=\sum_{k\in[K]}\vw_k\vs_k^H+\mW_A\mS_A^H,
\end{align}
where $\vs_k\in\mathbb{C}^{L \times 1}$ is the data symbol for the $k$-th communication user and $\mS_A\in\mathbb{C}^{L\times N_t}$ is the sensing signal, which are precoded by
the communication beamformer $\vw_k\in\mathbb{C}^{N_t\times 1}$ and the auxiliary beamforming matrix $\mW_A\in\mathbb{C}^{N_t\times N_t}$.
Assume that the data streams $\tilde{\mS}=[\vs_1, \vs_2, \ldots,\vs_K,\mS_A]^H$ are asymptotically orthogonal \cite{liu2021cramer} to each other for sufficiently large $L$, i.e., $\frac{1}{L}\tilde{\mS}\tilde{\mS}^H\approx\mI_{K+N_t}$.

\subsection{Problem Formulation}
\textbf{Communication metric.} For multi-user communications, by transmitting $\mX$ to $K$ users, the received signal $\vy_k$ of user $k$ is given as
\begin{align}\label{eqn:rsm}
\vy_k = \vh_k^H\mX+\vn_C,
\end{align}
where $\vh_k\in \mathbb{C}^{ N_t\times 1}$ is the communication  channel between the BS and user $k$, which is assumed to be known to the BS; $\vn_C$ is an additive white Gaussian noise (AWGN) vector with the variance of
each entry being $\sigma_C^2$.
 Then the SINR at the $k$-th communication user can be expressed as
 \begin{align}\label{eqn:SINR}
  \tilde{\gamma}_k=\frac{\left|\vh_{k}^{H} \vw_{k}\right|^{2}}{\sum_{i=1, i \neq k}^{K}\left|\vh_{k}^{H} \vw_{i}\right|^{2}+\left\|\vh_{k}^{H} \mW_A\right\|^{2}+\sigma_{C}^{2}}.
 \end{align}
 One of the most important criteria for multi-user beamforming is the overall system throughput \cite{weeraddana2012weighted}
\begin{align}\label{sum_rate}
\sum_{k\in[K]}\log(1+\tilde{\gamma}_k).
\end{align}

\textbf{Sensing metric.} By transmitting $\mX$ to sense the target, the reflected echo signal at the BS is given by
\vspace{-1mm}
\begin{align}\label{eqn: YR}
\mY_s=\mG\mX+\mN_s,
\end{align}
where $\mG\in \mathbb{C}^{N_r\times N_t}$ denotes the extended target response matrix and $\mN_s\in \mathbb{C}^{N_r\times L}$ is an AWGN matrix with the variance of each entry being  $\sigma_s^2$.
For the purpose of target sensing, we focus on estimating the response
matrix $\mG$. The CRB of estimating the response matrix is given by \cite{liu2021cramer}
\vspace{-1mm}
\begin{align}\label{eqn:CRB}
\textmd{CRB}(\mG) = \frac{\sigma_s^2N_r}{L}\tr(\mR_X^{-1}),
\end{align}
where
\vspace{-2mm}
\begin{align}\label{eqn:RX}
\mR_X =  \frac{1}{L}\mX\mX^H= \sum_{k\in[K]}\vw_k\vw_k^H+\mW_A\mW_A^H
\end{align}
is the sample covariance matrix of $\mX$ since the orthogonal data stream assumption.

Based on the sum-rate expression in \eqref{sum_rate} and the CRB  expression in \eqref{eqn:CRB}, the joint communication and sensing beamforming design problem can be formulated as
\begin{subequations}\label{beamform}
\begin{align}
\label{beamform_obj}&\min_{\substack{\{\vw_k\}_{k=1}^K,\mW_A}}~ -\sum_{k\in[K]}\log(1+\tilde{\gamma}_k)+ \rho\tr\left(\mR_X^{-1}\right) \\
\label{beamform_cons2} &\qquad~\textmd{s.t.} ~ ~~~~ \tr\left(\sum_{k\in[K]}\vw_k\vw_k^H+\mW_A\mW_A^H\right)\leq P_T,
\end{align}
\end{subequations}
where $\rho$ is a  positive parameter to trade-off the sum rate and  the CRB  and $P_T$ is the total transmit power budget
of the BS.

The problem \eqref{beamform} may strike a scalable trade-off between the communication and sensing performance by choosing different values of parameter $\rho.$
 This goal, however, cannot be achieved without the global solution of problem \eqref{beamform}.
In the rest part of this paper, we focus on designing global algorithms for solving this problem. It is worth highlighting that, although we formulate the problem as in \eqref{beamform}, the proposed algorithms in this paper can also be used to solve the other formulations of the joint communication and sensing beamforming design problem such as the SRM problem subject to the sensing CRB constraint on the target and the total power budget constraint of the BS.

\section{Theoretical Properties of the Optimal Solution to Problem  \eqref{beamform}}\label{sec_property}
In this section, we derive some important properties of the optimal solution for the joint beamforming design problem \eqref{beamform}. These properties play a central role in the development of an efficient global optimization algorithm  for solving problem \eqref{beamform} in Section IV. More specifically, we  propose a relaxation of problem \eqref{beamform} and show its tightness in Section III-A; then we derive the optimal solution of problem \eqref{beamform} in two special cases  which are the single-user case (i.e., $K=1$)  and the multi-user case where the channel vectors of different users are orthogonal.
\subsection{A Relaxation of Problem \eqref{beamform} and Its Tightness}
Introducing some auxiliary variables $\left\{\Gamma_k\right\}_{k=1}^K$,  we can reformulate problem \eqref{beamform} as
 \begin{subequations}\label{beamform_multi}
\begin{align}
\label{beamform_multi_obj}&\min_{\substack{\{\Gamma_k\}_{k=1}^K,\\
\{\vw_k\}_{k=1}^K,\mW_A}}~ -\sum_{k\in[K]}\log(1+\Gamma_k)+ \rho \tr\left(\mR_X^{-1}\right) \\
\label{beamform_multi_cons2} &\qquad\textmd{s.t.}  \qquad~~ \tr\left(\sum_{k\in[K]}\vw_k\vw_k^H+\mW_A\mW_A^H\right)\leq P_T,\\
 \label{beamform_multi_cons3}&\qquad\qquad\qquad\tilde{\gamma}_k\geq \Gamma_k,~ k\in [K],
\end{align}
\end{subequations}
where $\{\tilde{\gamma}_k\}_{k=1}^K$ and $\mR_X$  are given in  (\ref{eqn:SINR}) and (\ref{eqn:RX}), respectively.
Let
$
\mW_k=\vw_k\vw_k^H$ for all $k\in[K]$ and $\mW_{K+1}=\mW_A\mW_A^H
$. Then we get $\mW_k\succeq 0$ and
 $\textmd{rank}(\mW_k)=1$ for all $k\in [K].$
According to the  definition of $\tilde{\gamma}_k$ in \eqref{eqn:SINR}, the $k$-th SINR constraint \eqref{beamform_multi_cons3} can be rewritten as
\begin{align}\label{eqn:QW}
\tr(\mQ_k\mW_k)-\Gamma_k\sum_{i\neq k}^{K+1}\tr(\mQ_k\mW_i)\geq \Gamma_k\sigma_C^2,~ k\in [K],
\end{align}
where $\mQ_k=\vh_k\vh_k^H$ for all $k\in[K]$.  Then problem \eqref{beamform_multi} can be rewritten as
 \begin{subequations}\label{beamform_matrix}
\begin{align}
\label{beamform_matrix_obj}&\min_{\substack{\{\Gamma_k\}_{k=1}^{K},\{\mW_k\}_{k=1}^{K+1}}} -\sum_{k\in[K]}\log(1+\Gamma_k)+ \rho \tr\left(\mR_X^{-1}\right) \\
  \label{beamform_matrix_cons12} &\qquad~~~\textmd{s.t.} \qquad~~~ ~\sum_{k\in[K+1]}\tr\left(\mW_k\right)\leq P_T,~ \eqref{eqn:QW},\\
 \label{beamform_matrix_cons22}&\qquad\qquad\qquad\qquad~~\mW_k\succeq 0, ~k\in [K+1],\\
  \label{beamform_matrix_cons4}&\qquad\qquad\qquad\qquad~~\textmd{rank}(\mW_k)=1,  ~k\in [K].
\end{align}
\end{subequations}

By dropping all  rank-one constraints in \eqref{beamform_matrix_cons4}, problem \eqref{beamform_matrix} is relaxed into the following optimization problem
 \begin{subequations}\label{beamform_matrix1}
\begin{align}
\label{beamform_matrix_obj1}&\min_{\substack{\{\Gamma_k\}_{k=1}^{K},\{\mW_k\}_{k=1}^{K+1}}} -\sum_{k\in[K]}\log(1+\Gamma_k)+ \rho \tr\left(\mR_X^{-1}\right) \\
\label{beamform_matrix_obj12} &\qquad~~~\textmd{s.t.} \qquad~~~ ~\sum_{k\in[K+1]}\tr\left(\mW_k\right)\leq P_T,~ \eqref{eqn:QW},\\
\label{beamform_matrix_obj3} &\qquad\qquad\qquad\qquad~~\mW_k\succeq 0,  ~k\in [K+1].
\end{align}
\end{subequations}
Note that the above optimization problem \eqref{beamform_matrix1} is \emph{not} convex since there  exist many bilinear terms like $\Gamma_k\tr(\mQ_k\mW_i)$ in \eqref{eqn:QW}.
However, if $\Gamma_k$ for all $k\in [K]$ are fixed and given, then problem \eqref{beamform_matrix1} is equivalent to a convex problem considered in \cite[Eq. (36)]{liu2021cramer}.  This observation is important to the development of the global algorithm for solving problem \eqref{beamform} in the next section.

Now let us assume that an optimal solution of problem \eqref{beamform_matrix1} has been obtained and consider how to construct a rank-one solution directly from it.
From the definition of $\mR_X$ in \eqref{eqn:RX}, problem \eqref{beamform_matrix1} can be rewritten as
\vspace{-2mm}
 \begin{subequations}\label{beamform_multi1}
\begin{align}
\label{beamform_multi_obj1}&\min_{\substack{\mR_X,\{\Gamma_k\}_{k=1}^K,\\
\{\mW_k\}_{k=1}^K}} -\sum_{k\in[K]}\log(1+\Gamma_k)+ \rho \tr\left(\mR_X^{-1}\right) \\
\label{beamform_multi_cons12} &\qquad~\textmd{s.t.} ~~ \tr(\mR_X)\leq P_T,\mW_k\succeq 0,~ k\in [K],
\\
\label{beamform_multi_cons14}&\qquad~~~~ ~ ~ ~\mR_X\succeq \sum_{k\in[K]}\mW_k,\\
\label{beamform_multi_cons13} &\qquad~~~~ ~  ~~\tr(\mQ_k\mW_k)-\Gamma_k\tr(\mQ_k(\mR_X-\mW_k))\geq \Gamma_k\sigma_C^2,~ k\in [K].
\end{align}
\end{subequations}
 There are a lot of works  (e.g., \cite{liu2020joint,ma2017unraveling}) that study the tightness of SDRs in the context of beamformer design.  Following a similar argument in \cite{liu2020joint}, we can show that the relaxation in (\ref{beamform_multi1}) is tight in the sense that it has a rank-one solution.
\begin{proposition}\label{theorem_rank}
Given an optimal solution $\bar{\mR}_X$, $\{\bar{\Gamma}_k\}_{k=1}^K$, $\{\bar{\mW}_k\}_{k=1}^K$ of problem \eqref{beamform_multi1}, the following
$\tilde{\mR}_X$, $\{\tilde{\Gamma}_k\}_{k=1}^K$, $\{\tilde{\mW}_k\}_{k=1}^K$ is also its optimal solution:
\begin{align}\label{rrx}
\tilde{\mR}_X= \bar{\mR}_X, ~\tilde{\Gamma}_k = \bar{\Gamma}_k, ~ \tilde{\mW}_k=\frac{\bar{\mW}_k\mQ_k\bar{\mW}_k^H}{\tr(\mQ_k\bar{\mW}_k)},~k\in[K].
\end{align}
Moreover,  $\textmd{rank}(\tilde{\mW}_k)=1$  for all $k\in [K]$.
\end{proposition}
\begin{proof}
 See the supplementary material.
\end{proof}
According to  Proposition \ref{theorem_rank}, we can  find an optimal solution $\tilde{\mR}_X, \{\tilde{\Gamma}_k\}_{k=1}^K, \{\tilde{\mW}_k\}_{k=1}^K$   with all of $\tilde{\mW}_k$  being rank-one if an optimal solution $\bar{\mR}_X, \{\bar{\Gamma}_k\}_{k=1}^K,\{\bar{\mW}_k\}_{k=1}^K$  of problem \eqref{beamform_multi1} is obtained.
In addition, the optimal beamformer $\{\vw_k\}_{k=1}^K$ and $\mW_A$ for the original problem \eqref{beamform} is straightforwardly expressed as
\vspace{-3mm}
\begin{align*}
\vw_k&=\left(\vh_k^H\bar{\mW}_k\vh_k\right)^{-1/2}\bar{\mW}_k\vh_k, ~k\in [K],\\
\mW_A\mW_A^H&=\tilde{\mR}_X-\sum_{k\in[K]}\tilde{\mW}_k.
\end{align*}
Therefore, we only need to  focus on solving problem \eqref{beamform_multi1} in order to solve  problems \eqref{beamform} and \eqref{beamform_multi}.
\subsection{Closed-Form Solutions in Two Special Cases}
In this subsection, we derive an optimal closed-form solution to problem (\ref{beamform_multi1}) when there is only a single communication user or the channel vectors
of different users are orthogonal. We consider these two cases separately.
\subsubsection{Single-User Case}
In this part, we consider the case where $K=1$. Let
$\mQ_1=\vh_1\vh_1^H$ and $\mW_1=\vw_1\vw_1^H$. Then the optimization
problem \eqref{beamform_multi1} can be recast as
\begin{subequations}\label{single_user0}
\begin{align}
&\min_{\substack{\mW_1,\mR_X,\Gamma_1}}~ -\log\left(1+\Gamma_1\right)+ \rho \tr(\mR_X^{-1}) \\
\label{single_user0_1} &\qquad\textmd{s.t.}  ~~~~ \tr(\mR_X)\leq P_T,~ \mR_X \succeq \mW_1 \succeq 0,\\
\label{single_user0_2} &\qquad\qquad~~
 \tr(\mQ_1\mW_1)-\Gamma_1\tr(\mQ_1(\mR_X-\mW_1))\geq \Gamma_1\sigma_C^2.
\end{align}
\end{subequations}
Although problem \eqref{single_user0}   seems to be a nonconvex problem, we can derive its  optimal   closed-form solution   as in the following theorem.
\begin{theorem} \label{theorem_closed}
 The optimal solution of  problem \eqref{single_user0} is
 \begin{align}\label{WRX}
 \mW_1 =\frac{\Gamma_1\sigma_C^2\vh_1\vh_1^H}{\|\vh_1\|^4},
 ~\mR_X = \sum_{i\in[N_t]}  \lambda_{i}\vu_i\vu_i^H,
 \end{align}
 where
 \begin{align*}
 \lambda_{1}=\frac{\Gamma_1\sigma_C^2}{\|\vh_1\|^2},~ \lambda_{i} = \frac{P_T\|\vh_1\|^2-\Gamma_1\sigma_C^2}{\|\vh_1\|^2(N_t-1)}, ~i=2,\ldots,N_t,
 \end{align*}
 $\vu_1 = \vh_1/\|\vh_1\|$, and $\{\vu_i\}_{i=2}^{N_t}$ forms an orthogonal basis of the null space of $\vu_1$.
 Additionally, the optimal $\Gamma_1$  satisfies the following univariate equation:
 \begin{align}\label{gamma_equation}
\frac{\|\vh_1\|^2}{1+\Gamma_1}  = \rho\sigma_C^2\left(
\frac{(\|\vh_1\|^2(N_t-1))^2}{(P_T\|\vh_1\|^2-\Gamma_1\sigma_C^2)^2}-
\frac{\|\vh_1\|^4}{\Gamma_1^2\sigma_C^4}\right).
\end{align}
 \end{theorem}
 \begin{proof} See Appendix \ref{proof_theorem2}.
  \end{proof}
  We can observe from \eqref{WRX} that the optimal closed-form solution given in Theorem \ref{theorem_closed}
satisfy the   rank-one constraints, i.e., $\textmd{rank}(\mW_k)=1$ for all $k\in[K]$. Hence it is also
the optimal solution to   problem \eqref{beamform} in the single-user case. Compared with \cite[Theorem 3]{liu2021cramer},  $\Gamma_1$ here is a variable    rather than a given constant and   the optimal $\Gamma_1$ must satisfy the equation \eqref{gamma_equation}.
 \subsubsection{Multi-User Orthogonal Case}
In the general multi-user case, it is difficult to obtain the closed-form solution for problem \eqref{beamform_multi1}, since there exist nonconvex terms in \eqref{beamform_multi_cons13}. If all interference terms are zero, then the nonconvex constraints \eqref{beamform_multi_cons13} are convex. In the following lemma, we show that all interference terms are indeed zero at the optimal solution of problem (\ref{beamform_multi1}) when the channels of different users are orthogonal.
 \begin{lemma}\label{multi_lemma}
 Suppose that the communication channels $\left\{\vh_k\right\}_{k=1}^K$ are orthogonal to each other, i.e., $\vh_i^H\vh_j=0$   for all $i \neq j$, $i,~j\in[K]$. Then there exists an optimal solution  $\mR_X$  and $\{\mW_k\}_{k=1}^K$ such that all interference terms are minimized to  be zero, i.e.,
\begin{align}\label{inter_Qk}
\tr(\mQ_k(\mR_X-\mW_k))=0,~k\in[K].
\end{align}
 In addition, let $\lambda_{k}=\vu_k^H\mR_X\vu_k$  and $\vu_k={\vh_k}/{\|\vh_k\|}$  for all $k\in [K]$. Then $\lambda_{k}$ and $\vu_k$   is an eigenpair of the optimal $\mR_X$ of problem \eqref{beamform_multi1} for all $k\in[K]$.
   \end{lemma}
Since the proof of Lemma \ref{multi_lemma} is essentially the same as that of Theorem \ref{theorem_closed}, we provide the proof in the supplemental material for completeness.

Under the condition of Lemma \ref{multi_lemma}, the optimization problem \eqref{beamform_multi1}  can   be simplified into
\begin{subequations}\label{multi_user0}
\begin{align}
\label{multi_user1}\min_{\substack{\{\lambda_{i}\}_{i=1}^{N_t},\{\Gamma_k\}_{k=1}^K}}&~ -\sum_{k\in[K]}\log\left(1+\Gamma_k\right)+ \rho \sum_{i\in[N_t]}\lambda_{i}^{-1} \\
 \textmd{s.t.}~~~~~~~ & ~~~~ \lambda_{i}\|\vh_i\|^2\geq \Gamma_i\sigma_C^2, ~  i\in[K],\\
 &~~~\sum_{i\in[N_t]}\lambda_{i}\leq P_T,~\lambda_{i}> 0,~ i\in[N_t].
\end{align}
\end{subequations}
Obviously, problem \eqref{multi_user0} is a convex optimization problem, which can be solved efficiently and globally   (e.g., by CVXPY \cite{diamond2016cvxpy}).

The orthogonal condition  on the communication channel vectors is difficult to satisfy in practice.  However, if
  $\{\vh_k\}_{k=1}^K$ are isotropic and independent random channels, then they tend to be
almost orthogonal in high-dimensional spaces with high probability (see \cite[Lemma 3.2.4]{vershynin2018high}).
\vspace{-2mm}
 \section{Proposed Branch-and-bound Algorithm}\label{sec_BB}
 In this section, we propose a global optimization algorithm based on the B\&B scheme for solving the nonconvex problem \eqref{beamform_multi1} in the general multi-user case.
 Since the efficiency of the B\&B algorithm relies on the quality of the lower bound, we are motivated to find  a tight convex relaxation of the nonconvex set defined by \eqref{beamform_multi_cons13}.  To do so, we first present a convex relaxation of problem (\ref{beamform_multi1}) based on the McCormick  envelope in Section IV-A. Then we propose the global B\&B algorithm for solving problem (\ref{beamform_multi1}) in Section IV-B.
\vspace{-3mm}
\subsection{An McCormick Envelope based Relaxation}
Let us introduce an auxiliary variable $a_k$ for each $k\in[K]$. Then the $k$-th nonconvex constraint in \eqref{beamform_multi_cons13} is rewritten as
\begin{align}
\nonumber &\tr(\mQ_k\mW_k)-a_k\geq \Gamma_k\sigma_C^2, \\
\label{eqn_sk} &a_k = \Gamma_k\tr(\mQ_k(\mR_X-\mW_k)).
\end{align}
One can observe that  there is still a  bilinear term in \eqref{eqn_sk}.
\begin{figure}[t]
\vspace{-0.5cm}
\begin{center}
\includegraphics[width=0.7\linewidth]{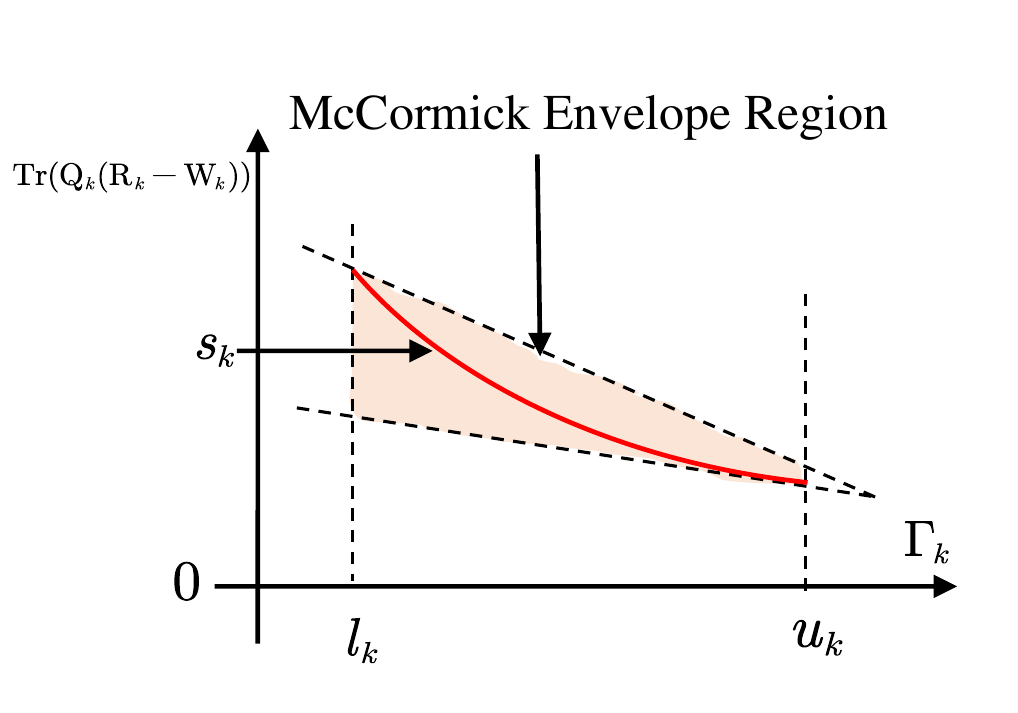}
\vspace{-0.5cm}
\caption{An illustration of the convex McCormick envelope.}
\label{McCormick_Envelopes}
\end{center}
\vspace{-0.8cm}
\end{figure}
Next, we develop a convex relaxation for \eqref{eqn_sk} based on the McCormick envelope \cite{mitsos2009mccormick}.
\begin{lemma}\label{lemma_ebve}
 Given any $k\in[K],$ assume there exist $\ell_k, u_k,$ and $b_k$ such that $\ell_k\leq\Gamma_k\leq u_k$ and $0\leq\tr(\mQ_k(\mR_X-\mW_k))\leq b_k$. Then the McCormick envelope for the bilinear constraint \eqref{eqn_sk} is
\vspace{-4mm}
\begin{subequations}\label{enve_lopes}
\begin{align}
\label{enve_lopes1} a_k \geq& ~\ell_k\tr(\mQ_k(\mR_X-\mW_k)),\\
\label{enve_lopes2} a_k \geq&~ u_k\tr(\mQ_k(\mR_X-\mW_k))+(\Gamma_k-u_k)b_k,\\
\label{enve_lopes3} a_k \leq&~ u_k\tr(\mQ_k(\mR_X-\mW_k)),\\
\label{enve_lopes4} a_k \leq&~ (\Gamma_k-\ell_k)b_k+\ell_k\tr(\mQ_k(\mR_X-\mW_k)),
\end{align}
\end{subequations}
all of which are linear constraints with respect to $\mR_X, a_k, \Gamma_k$, and $\mW_k$.
\end{lemma}
 An illustration for the linear constraints (\ref{enve_lopes1})-(\ref{enve_lopes4})
and the McCormick envelope region is given in Fig. \ref{McCormick_Envelopes}. From
 the constraint $\tr(\mR_X)\leq P_T$, it is simple to obtain
\begin{align*}
0\leq\Gamma_k\leq \frac{P_T\|\vh_k\|^2}{\sigma_C^2} ~\text{and}~ \tr(\mQ_k(\mR_X-\mW_k))\leq P_T\|\vh_k\|^2.
\end{align*}
The above lower and upper bounds on the SINR and interference terms provide desirable bounds in Lemma \ref{lemma_ebve}.

Setting $b_k=P_T\|\vh_k\|^2$ in the above Lemma \ref{lemma_ebve}, we immediately obtain the following convex McCormick envelope based relaxation (MER) of problem \eqref{beamform_multi1}:
\vspace{-1mm}
 \begin{subequations}\label{MER_multi1}
\begin{align}
\label{MER_multi_obj1}&\min_{\substack{\mR_X,\{\Gamma_k\}_{k=1}^K,\\
\{\mW_k\}_{k=1}^K,\{a_k\}_{k=1}^K}} \overbrace{-\sum_{k\in[K]}\log(1+\Gamma_k)+ \rho \tr\left(\mR_X^{-1}\right)}^{\Phi\left(\mR_X, \{\Gamma_k\}_{k=1}^K, \{\mW_k\}_{k=1}^K\right)}\\[-2mm]
\label{MER_multi_cons12} &\qquad~~\textmd{s.t.}~  \tr(\mR_X)\leq P_T, ~\mR_X\succeq \sum_{k=1}^{K}\mW_k,~ \mW_k\succeq 0,\\
\label{MER_multi_cons13}&\qquad\qquad~\tr(\mQ_k\mW_k)-a_k\geq \Gamma_k\sigma_C^2,~\eqref{enve_lopes}, \\
\label{MER_multi_cons123}& \qquad\qquad~0\leq\tr(\mQ_k(\mR_X-\mW_k))\leq b_k,~ k\in[K],\\
\label{MER_multi_cons14}&\qquad\qquad~\ell_k\leq\Gamma_k\leq u_k,~ k\in[K].
\end{align}
\end{subequations}
 The above problem \eqref{MER_multi1} is a convex problem that can be solved efficiently and globally via CVXPY \cite{diamond2016cvxpy} and its optimal value provides a lower bound for  the optimal value of problem \eqref{beamform_multi1}. Obviously,  the quality of this lower bound heavily depends on the choices of the rectangle set $\prod_{k=1}^K[\ell_k, u_k]$.


\subsection{Proposed B\&B Algorithm}
Now we are ready to present our proposed B\&B algorithm for globally solving problem \eqref{beamform_multi1}. The basic idea of the proposed
algorithm is to relax the original  nonconvex problem \eqref{beamform_multi1} with bilinear constraints to the convex MER \eqref{MER_multi1} and gradually tighten the relaxation by reducing the width of the associated intervals $[\ell_k,u_k]$ for the $k$-th communication user.  The B\&B algorithm uses a tree search strategy to store all of the above MER relaxation subproblems defined over different feasible regions as well as their solutions and  to implicitly enumerate them to find the global solution of the original problem.


For ease of presentation, we first introduce some notations. Let MER($\mathcal{Q}$) denote  the corresponding MER problem defined over the rectangle set $\mathcal{Q}:=\prod_{k=1}^K[\ell_k,u_k]$; let $L$ be the optimal value of MER($\mathcal{Q}$) and $\mV:=\left[\mR_X,\{\Gamma_k\}_{k=1}^K,\{\mW_k\}_{k=1}^K,\{a_k\}_{k=1}^K\right]$  be a collection of  the optimal solutions of all variables; let  $\mathcal{P}$ denote the constructed problem list   of all unbranched nodes and $\{\mathcal{Q},\mV, L\}$ denote a problem from
the list $\mathcal{P}$; let $t$ denote the iteration index
of the B\&B algorithm, where an iteration corresponds to a
branching operation;  let  $U^t$ and $\bar{\mV}^t$ denote the upper bound and the best known feasible solution at the $t$-th iteration, respectively; let $\mV^*$ denote the best known feasible solution and $U^*$ denote the
objective value of problem \eqref{beamform_multi1} at $\mV^*$.
Then we can present the following key components of the proposed B\&B algorithm.
%

 \emph{Initialization}: We initialize all intervals $[\ell_k^0,u_k^0]$ for all $k\in[K]$ to be $\left[0, P_T\|\vh_k\|^2/\sigma_C^2\right]$ for the MER problem \eqref{MER_multi1} and $\mathcal{Q}^0:= \prod_{k=1}^K[\ell_k^0,u_k^0]$. In this
case, problem \eqref{MER_multi1} reduces to the convex optimization problem MER($\mathcal{Q}^0$) as follows:
 \begin{subequations}\label{MER_multi_re0}
\begin{align}
\label{MER_multi_re01}&\min_{\mV } ~\Phi\left(\mR_X, \{\Gamma_k\}_{k=1}^K,\{\mW_k\}_{k=1}^K\right)\\
\label{MER_multi_re02}&~\textmd{s.t.}~~ \eqref{MER_multi_cons12}-\eqref{MER_multi_cons123},~ \Gamma_k\in \mathcal{Q}^0,~k\in[K].
\end{align}
\end{subequations}
Denote its optimal solution and optimal value by $\mV^0$ and $L^0$, respectively.

 \emph{Termination}: Let $\{\mathcal{Q}^t,\mV^t, L^t\}$ denote the problem instance that has the least lower bound in the problem list $\mathcal{P}$. Given an error tolerance $\epsilon$, if
 \vspace{-1mm}
\begin{align}\label{eqn:term}
U^t-L^t\leq \epsilon,
\end{align}
we stop the algorithm; otherwise we branch one interval in $\mathcal{Q}^t$ as specified below in \eqref{eqn:k}.

 \emph{Branch}: Suppose that the stopping criterion in \eqref{eqn:term} is not satisfied, we select one interval that leads to the largest relaxation gap to be branched  into two smaller sub-intervals. Let $\mV^t$ be the optimal solution of problem MER($\mathcal{Q}^t$). Since problem MER($\mathcal{Q}^t$) is a relaxation of problem \eqref{beamform_multi1}, its solution might not satisfy the constraint \eqref{beamform_multi_cons13}. Fortunately, we can construct a feasible solution $\hat{\mV}^t:=\left[\hat{\mR}_X,\{\hat{\Gamma}_k\}_{k=1}^K,\{\hat{\mW}_k\}_{k=1}^K,\{\hat{a}_k\}_{k=1}^K\right]$ to problem \eqref{beamform_multi1} based on the solution  $\mV^t$ of MER($\mathcal{Q}^t$) as follows:
 \vspace{-2mm}
     \begin{subequations}\label{eqn:feas}
    \begin{align}
    \hat{\mR}_X^t&= \mR_X^t,~\hat{\mW}_k^t=\mW_k^t,~ \hat{a}_k^t = a_k^t,~k\in[K],\\
\hat{\Gamma}_k^t& =  \frac{\Gamma_k^t\sigma_C^2+\ell_k^t\tr(\mQ_k(\mR_X^t-\mW_k^t))}{\sigma_C^2+\tr(\mQ_k(\mR_X^t-\mW_k^t))}, ~k\in[K].
\end{align}
\end{subequations}
Substituting  \eqref{eqn:feas} into the constraints \eqref{beamform_multi_cons12}--\eqref{beamform_multi_cons13}, we can observe that $\hat{\mV}^t$ is a feasible solution to problem \eqref{beamform_multi1}.
 The  constructed solution in \eqref{eqn:feas} plays a central role in improving the upper bound in the  B\&B algorithm and in selecting the user that leads to the largest relaxation gap to be branched.

In particular, we use the following rule to  select the user that has the largest relative relaxation gap:
\vspace{-2mm}
\begin{align}\label{eqn:k}
 k^* &=\mathop{\arg\max}\limits_{k\in [K]} \left\{ \frac{\Gamma_k^t-\hat{\Gamma}_k^t}{1+\hat{\Gamma}_k^t}\right\}.
\end{align}
It is clear that the numerator in \eqref{eqn:k} is the gap between the predicted SINR (by the relaxation) and the practically achieved SINR of user $k$ and hence the quantity in \eqref{eqn:k} measures the relative relaxation gap between the predicted and achieved SINRs of user $k.$

 Next,
 we  partition $\mathcal{Q}^t$ into two sets (denoted as $\mathcal{Q}_1^t$ and $\mathcal{Q}_2^t$) by partitioning its $k^*$-th interval into two equal intervals and keep all the others being unchanged. Specifically,
 \begin{align}
\label{eqn_q1}\mathcal{Q}_1^t&=\{\Gamma_k\in \mathcal{Q}^t\mid\Gamma_{k^{*}}\leq z_{k^*}^t\},\\
\label{eqn_q2}\mathcal{Q}_2^t&=\{\Gamma_k\in \mathcal{Q}^t\mid\Gamma_{k^{*}}\geq z_{k^*}^t\},
\end{align}
 where $z_{k^*}^t={(\ell_{k^*}^t+u_{k^*}^t)}/{2}$.
 Then we solve the MER subproblems defined over the two newly obtained  smaller sets $\mathcal{Q}_1^t$ and $\mathcal{Q}_2^t$, which are called children problems. Obviously, the two children problems obtained from partitioning $\mathcal{Q}^t$ are tighter than the one defined over the original set $\mathcal{Q}^t$. In this way, the B\&B process gradually tightens the relaxations and is able to find a (nearly) global solution satisfying the condition in \eqref{eqn:term}.
    When $\mathcal{Q}^t$ has been branched into two sets, the MER problem defined over $\mathcal{Q}^t$ will be deleted from the problem list $\mathcal{P}$, and the two corresponding children problems   defined over  $\mathcal{Q}_1^t$ and $\mathcal{Q}_2^t$ will be added into $\mathcal{P}$.

\emph{Lower Bound}: For any problem instance $\{\mathcal{Q},\mV, L\}$, $L$ is the lower bound of the optimal value of the original nonconvex problem \eqref{beamform_multi1} defined over $\mathcal{Q}$.
Therefore, the smallest one among all bounds is a lower bound of the optimal value of the original problem \eqref{beamform_multi1}.  At the $t$-th iteration, we choose a problem instance from $\mathcal{P}$, denoted as $\{\mathcal{Q}^t,\mV^t,L^t\}$, such that the bound $L^t$ is the smallest one in $\mathcal{P}$.

\emph{Upper Bound}: An upper bound is obtained from the best known feasible solution of \eqref{beamform_multi1}.
Since $\hat{\mR}_X^t$, $\{\hat{\Gamma}_k^t\}_{k=1}^K$, $\{\hat{\mW}_k^t\}_{k=1}^K$ is a feasible solution for problem \eqref{beamform_multi1}, then
\begin{align}\label{eqn:upper_obj}
\hat{U}^t:=\Phi\left(\hat{\mR}_X^t, \{\hat{\Gamma}_k^t\}_{k=1}^K,\{\hat{\mW}_k^t\}_{k=1}^K\right)
\end{align}
is an upper bound of the original problem. In our proposed algorithm, the upper
bound $U^t$ is the best objective values at all of the known feasible
solutions at the $t$-th iteration.

\begin{algorithm}[t]
\begin{algorithmic}[1]
\STATE {\bf Input} Problem instance $\{\vh_k\}_{k=1}^K$,  error tolerance $\epsilon>0$.
\STATE {\bf Initialization}  Solve problem MER$(\mathcal{Q}^0)$ in \eqref{MER_multi_re0} to obtain  $\{\mathcal{Q}^0,\mV^0,L^0\}$ and add it into the problem list $\mathcal{P}$. Compute a feasible point $\hat{\mV}^0$ and an upper bound
$U^0=\hat{U}^0$ by \eqref{eqn:feas} and \eqref{eqn:upper_obj}, respectively.
\FOR{ $t=0,1,\ldots$}
\IF{$U^t-L^t<\epsilon$}
\STATE terminate the algorithm, set $U^*=U^t$ and $\mV^*=\bar{\mV}^{t}$.
\ENDIF
\STATE  Choose the problem that has the lowest bound from $\mathcal{P}$ and
delete it from $\mathcal{P}$.
\STATE Obtain user index $k^*$ by \eqref{eqn:k}.
\STATE Branch $\mathcal{Q}^t$ into two sets $\mathcal{Q}_1^t$ and $\mathcal{Q}_2^t$  in \eqref{eqn_q1} and \eqref{eqn_q2}.
\FOR{ $j=1,2$}
\STATE Solve MER problem \eqref{MER_multi1}  defined over set $\mathcal{Q}_j^t$  to obtain its optimal solution $\mV_j^t$ and lower bound $L_j^t$.
\STATE Compute  $\hat{\mV}_j^t$ and  $\hat{U}_j^t$ by \eqref{eqn:feas} and \eqref{eqn:upper_obj}, respectively.
\STATE Add $\{\mathcal{Q}_j^t,\mV_{j}^t,L_j^t\}$ into the problem list $\mathcal{P}$.
\ENDFOR
\STATE Compute $j^*=\arg\min_{j\in\{1,2\}}\{\hat{U}_j^t\}$.
\IF{$U^t\geq \hat{U}_{j^*}^t$}
\STATE set $U^t=\hat{U}_{j^*}^t$ and $ \bar{\mV}^t = \hat{\mV}_{j^*}^t$.
\ENDIF
\ENDFOR
\STATE {\bf Return $\mV^*$ and $U^*$.}
\end{algorithmic}
\caption{Proposed B\&B Algorithm Solving Problem \eqref{beamform_multi1}}
\label{alg_BB}
\end{algorithm}
The pseudo-codes of our proposed  B\&B algorithm are given in Algorithm \ref{alg_BB}. To the best of our knowledge, our proposed algorithm is the first global algorithm for solving problem \eqref{beamform_multi1}.

\subsection{Global Optimality of Proposed B\&B Algorithm}
In this subsection,  we show that the proposed B\&B algorithm is guaranteed to find the global solution of problem \eqref{beamform_multi1}. Before presenting the theoretical results of the proposed algorithm, let us first define the
$\epsilon$-optimal solution of problem \eqref{beamform_multi1}.

\begin{definition}
 Given any $\epsilon>0$, a feasible point $\hat{\mR}_X$, $\{\hat{\Gamma}_k\}_{k=1}^K$, $\{\hat{\mW}_k\}_{k=1}^K$ is called an $\epsilon$-optimal solution of problem \eqref{beamform_multi1} if it satisfies
\begin{align}\label{eop_solution}
\Phi\left(\hat{\mR}_X,\{\hat{\Gamma}_k\}_{k=1}^K,\{\hat{\mW}_k\}_{k=1}^K\right)-U^*\leq \epsilon,
\end{align}
where $U^*$ is the
optimal value of problem \eqref{beamform_multi1}.
\end{definition}

Since $L^t$ is the smallest lower bound   in problem list $\mathcal{P}$ at the $t$-th iteration   according to line 7 of Algorithm \ref{alg_BB}, it follows that
$L^t$ is  less than or equal to the optimal value of problem \eqref{beamform_multi1}, i.e., $L^t\leq U^*$ for all $t\geq 1$. Then we immediately obtain
\begin{align*}
U^t-U^*\leq U^t-L^t, ~\forall~ t\geq 1,
\end{align*}
 which further implies that, if the proposed algorithm terminates  (i.e., \eqref{eqn:term} is satisfied), the returned solution
$\hat{\mR}_X^t$, $\{\hat{\Gamma}_k^t\}_{k=1}^K$,  and $\{\hat{\mW}_k^t\}_{k=1}^K$ by the proposed
algorithm is an $\epsilon$-optimal solution of problem \eqref{beamform_multi1}. The following lemma shows that the proposed Algorithm \ref{alg_BB} will terminate.
\begin{lemma}\label{lemma_con}
  For any given $\epsilon > 0$ and any given instance of problem (\ref{beamform_multi1}) with $K$ users, define
\begin{align}\label{eqn:delta}
\delta_{\epsilon} = \frac{1}{2}\left(\exp\left(\frac{\epsilon}{K}\right)-1\right).
\end{align}
 If
\begin{align}\label{lemma_cond}
u_{k^*}^t-\ell_{k^*}^t\leq 2\delta_{\epsilon},
\end{align}
where $k^*$ is obtained in (\ref{eqn:k}),
then the proposed Algorithm \ref{alg_BB} will terminate.
\end{lemma}
\begin{proof} See Appendix \ref{proof_lemma_con}.
\end{proof}
Now we present the main result of this subsection.
\begin{theorem}\label{theorem_global}
For any given $\epsilon>0$ and any given instance of problem (\ref{beamform_multi1}) with $K$ users, Algorithm \ref{alg_BB} will return an $\epsilon$-optimal solution within
\begin{align}\label{gmax}
T_{\epsilon}:=\left\lceil\left(\frac{\Gamma_{\max}}{\delta_{\epsilon}}\right)^{K}\right\rceil +1
\end{align}
iterations, where $\Gamma_{\max}= \mathop{\max}\limits_{k\in [K]}{P_T\|\vh_k\|^2}/{\sigma_C^2}$, $\delta_{\epsilon}$ is defined in \eqref{eqn:delta}, and $\lceil\cdot\rceil$ denotes the rounding operator.
\end{theorem}
\begin{proof}
See Appendix \ref{proof_theorem_global}.
\end{proof}
 Theorem \ref{theorem_global} shows that
the total number of iterations for our proposed    B\&B algorithm to return an $\epsilon$-optimal solution grows exponentially fast with the total number of users $K.$ The iteration complexity of the proposed  B\&B algorithm  in (\ref{gmax}) seems to be prohibitively high. However, our simulation results  in Section VI show that its practical iteration complexity is actually significantly less than the worst-case bound in Theorem \ref{theorem_global}, thanks to the effective lower bound provided by the relaxation problem \eqref{MER_multi1} based on the McCormick envelope relaxation. In the next section, we shall employ the ML technique to further accelerate the proposed B\&B algorithm.

\section{GNN-based Accelerated B\&B Algorithm}\label{sec_ml_bb}
The search procedure of the proposed B\&B algorithm essentially forms  a series of sequential decision  problems in a tree structure. In particular, at each  iteration of the B\&B algorithm, the pruning policy  makes a decision of pruning or preserving the node. Therefore,
 the node pruning policy plays an essential role in the computational efficiency of the B\&B algorithm in the sense that, if the node pruning policy can safely and quickly prune the corresponding  irrelevant node, (i.e., the node that does not contain the optimal solution), then the problem at the node and all of its children problems do not need to be explored and solved in the B\&B algorithm. The goal of this section is to leverage the ML technique to develop an effective pruning policy which can quickly and safely prune the irrelevant nodes in the B\&B algorithm, thereby significantly accelerating its convergence speed.

 More specifically, we first give the imitation learning framework that is used to accelerate our B\&B algorithm, where the pruning policy is modeled as a binary classifier (i.e., prune or preserve) and learned via imitation learning in Section V-A. Then in Section V-B, we propose a GNN-based node classifier in detail. Finally in Section V-C, we present the GNN-based accelerated B\&B algorithm.
\subsection{An Imitation Learning Framework}
Let  $\phi_s\in \mathbb{R}^P$ represent the mapping from a node $s$ to its feature representation,  where $P$ is the dimension of the feature vector of the node. Denote $\pi_{\theta}: \mathbb{R}^P \rightarrow [0,1]$ as the pruning policy parameterized by $\theta$.  After training the parameter $\theta,$ we can use $\pi_{\theta}(\phi_s)$ to determine whether node $s$ can be pruned. In particular, if $\pi_{\theta}(\phi_s) < 0.5$, then node $s$ is considered irrelevant; otherwise, the node is branched.

\textbf{Labeling data.} The training data $\{\phi_s, y_s\}_{s=1}^T$   are obtained in a batch-by-batch manner with online optimization,  where $y_s$ is the label of  the corresponding feature. In our case, the label $y_s$ can be determined according to the following rule:
\begin{align}\label{eqn_label}
y_s=\left\{
      \begin{array}{ll}
        1, & \hbox{if $\ell_k^s\leq\Gamma_k^{*}\leq u_k^s$ for all ~$k\in[K]$;} \\
        0, & \hbox{otherwise,}
      \end{array}
    \right.
\end{align}
where $\Pi_{k=1}^K[\ell_k^s, u_k^s]$  denotes the feasible region for variable $\{\Gamma_k\}_{k=1}^K$ at node $s$ and $\{\Gamma_k^{*}\}_{k=1}^K$ is the optimal solution of the corresponding instance returned by Algorithm \ref{alg_BB}.

\textbf{The imitation learning framework.} Since the node   generation process  in the  proposed Algorithm \ref{alg_BB} is sequential and   depends on the node pruning policy, it inspires us to adopt the imitation learning approach  to train the parameter $\theta$ \cite{he2014learning}. The imitation learning criterion is given as follows:
\begin{align}\label{IL_equ}
\theta^{i+1} = \arg\min_{\theta}\frac{1}{i}\sum_{t=1}^{i}\frac{1}{|\mathcal{D}_t|}\sum_{(\phi_s,y_s)\in \mathcal{D}_t} \mathcal{L}(\pi_{\theta}(\phi_s),y_s),
\end{align}
where $\mathcal{D}_t$ is the $t$-th batch of training pairs and $\mathcal{L}(\cdot,\cdot)$ is a binary classification loss.  In optimization problem \eqref{IL_equ}, the loss function is minimized iteratively  to improve the pruning
policy by using the aggregated dataset $\cup_{t=1}^i \mathcal{D}_t$. Over time this process drives
 the learned pruning rule $\pi_{\theta}(\cdot)$ to imitate the behaviors of the proposed B\&B Algorithm \ref{alg_BB}.
\subsection{GNN-based Pruning Policy}
In the context of wireless communications, where the number of   communication users may dynamically change, it is desirable to design a neural network that remains agnostic to such variations.   This motives us to propose a GNN-based node pruning policy.

We begin by defining a graph   for each node (in the B\&B algorithm), which sits in the lower right corner in Fig. \ref{BB_framework}, where the vertices represent the antennas and the users and the edges represent the channels connecting them. The input features of the vertices and edges play a crucial role in conveying the essential information.
In particular, for each node, let $\vx_n\in \mathbb{R}^{V_a}$, $n\in [N_t]$, $\vx_{N_t+k}\in \mathbb{R}^{V_u}$, $k\in [K]$, and $\ve_{n,N_t+k}\in \mathbb{R}^{V_e}, n\in [N_t], k\in [K]$
represent the feature vector of antenna $n$, user $k$, and the channel between antenna $n$ and user $k$, respectively.  The specific form for the features of $\vx_n$ and $\ve_{n,N_t+k}$ is given in Appendix \ref{feature_design}.

\begin{figure}[t]
\begin{center}
\includegraphics[width=1\linewidth]{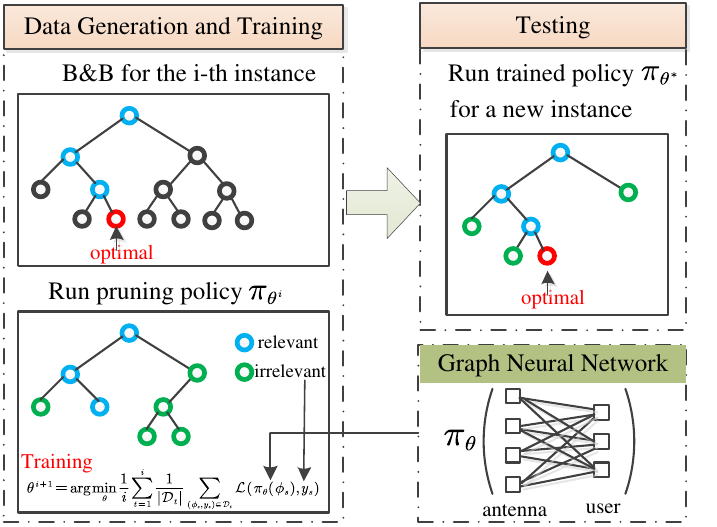}
\caption{A framework of the GNN-based accelerated B\&B algorithm.}
\label{BB_framework}
\end{center}
\vspace{-0.4cm}
\end{figure}
\begin{algorithm}[t]
\begin{algorithmic}[1]
\STATE {\bf Input} Training number $I$, instance number $R$, $D_1=\{\}$.
\FOR{ $i=1,2,\ldots, I$}
\FOR{ $r=1,2,\ldots, R$ (in parallel)}
\STATE Randomly generate  a problem instance $\{\vh_k\}_{k=1}^K$.
\STATE Run Algorithm \ref{alg_BB} to obtain  its optimal solution $\mV^*$.
\STATE Execute lines 3 to 21 in Algorithm \ref{alg_abb}.
\IF{\eqref{eqn_label} holds}
\STATE $D_i\leftarrow D_i\cup\{\phi_{s}, 1\}$,
\ELSE
\STATE $D_i\leftarrow D_i\cup\{\phi_{s}, 0\}$,
\ENDIF
\ENDFOR
\STATE Obtain $\theta^{i+1}$ by solving problem \eqref{IL_equ}.
\ENDFOR
\STATE {\bf Return} $\theta^*=\theta^{I}$.
\end{algorithmic}
\caption{Data Generation and Training}
\label{alg_dgt}
\end{algorithm}
\begin{algorithm}[t]
\begin{algorithmic}[1]
\STATE {\bf Input} Problem instance $\{\vh_k\}_{k=1}^K$, trained pruning
policy $\pi_{\theta}$,  error tolerance $\epsilon>0$.
\STATE {\bf Initialization}  Solve problem MER$(\mathcal{Q}^0)$ in \eqref{MER_multi_re0} to obtain  $\{\mathcal{Q}^0,\mV^0,L^0\}$ and add it into the problem list $\mathcal{P}$. Compute a feasible point $\hat{\mV}^0$ and an upper bound
$U^0=\hat{U}^0$ by \eqref{eqn:feas} and \eqref{eqn:upper_obj}, respectively.
\FOR{ $t=0,1,\ldots$}
\IF{$U^t-L^t<\epsilon$}
\STATE terminate the algorithm, set $U^*=U^t$ and $\mV^*=\bar{\mV}^{t}$.
\ENDIF
\STATE  Choose the node $s$ that has the lowest bound from $\mathcal{P}$ and
delete it from $\mathcal{P}$.
\IF{$\pi_{\theta}(\phi_s)\geq 0.5$}
\STATE Obtain user index $k^*$ by \eqref{eqn:k}.
\STATE Branch $\mathcal{Q}^t$ into two sets $\mathcal{Q}_1^t$ and $\mathcal{Q}_2^t$  in \eqref{eqn_q1} and \eqref{eqn_q2}.
\FOR{ $j=1,2$}
\STATE Solve MER problem \eqref{MER_multi1}  defined over set $\mathcal{Q}_j^t$  to obtain its optimal solution $\mV_j^t$ and lower bound $L_j^t$.
\STATE Compute  $\hat{\mV}_j^t$ and  $\hat{U}_j^t$ by \eqref{eqn:feas} and \eqref{eqn:upper_obj}, respectively.
\STATE Add $\{\mathcal{Q}_j^t,\mV_{j}^t,L_j^t\}$ into the problem list $\mathcal{P}$.
\ENDFOR
\STATE Compute $j^*=\arg\min_{j\in\{1,2\}}\{\hat{U}_j^t\}$.
\IF{$U^t\geq \hat{U}_{j^*}^t$}
\STATE set $U^t=\hat{U}_{j^*}^t$ and $ \bar{\mV}^t = \hat{\mV}_{j^*}^t$.
\ENDIF
\ENDIF
\ENDFOR
\STATE {\bf Return $\mV^*$ and $U^*$.}
\end{algorithmic}
\caption{GNN-based Accelerated B\&B Algorithm}
\label{alg_abb}
\end{algorithm}
In this paper, we use the message passing GNN \cite{gilmer2017neural}, where the aggregation is done as follows:
\begin{align*}
\vq_r^{d}=\xi\left(\mZ_1\vq_r^{d-1}+\sum_{v\in \mathcal{E}_r}\left(\mZ_2\vq_v^d + \mZ_3\ve_{r,v}\right)\right).
\end{align*}
 In the above, $\vq_r^0=\vx_r,  r\in [N_t+K]$; $\mZ_i$ for $i=1,2,3$ are the aggregation parameters that should be learned; $\xi(\cdot)$ represents the activation function of layer $d$; $\mathcal{E}_r$ represents  the neighboring index set of vertex $r$ in the graph. The output of the GNN is
\begin{align*}
\pi_{\theta}(\phi_s)=\frac{1}{M}\sum_{r\in[M]}\zeta(\beta^T\vq_r^D),
\end{align*}
where  $\zeta(\cdot)$ is a sigmoid function and $\phi_s=[\vq_1^D, \vq_2^D, \ldots, \vq_M^D]\in \mathbb{R}^P$  represents the feature vector at the final layer $D$. Finally, all parameters in the GNN to be optimized   are  compactly written as $\theta=[\mZ_1,\mZ_2,\mZ_3,\beta]$.

%
\vspace{-2mm}
\subsection{Proposed GNN-based Accelerated B\&B Algorithm}
We employ an imitation learning framework to train  the parameters in the GNN, as outlined in Algorithm \ref{alg_dgt} and illustrated in the left of Fig. \ref{BB_framework}. The   training algorithm consists of two main steps at each iteration,  i.e.,  data collection step and classifier improvement step,   which sit at the top and bottom of the left of Fig. \ref{BB_framework}, respectively.

More specifically, to generate the dataset $\mathcal{D}_i$, we first run the  vanilla B\&B algorithm  (i.e., Algorithm 1) on $R$ problem instances to find their optimal solutions.  Then we run the B\&B algorithm equipped with the current pruning policy $\pi_{\theta^i}$ to solve the previous  $R$ problem instances and, for each of them, we can  generate a tree  with many nodes as well as the feature vectors associated with the nodes. Finally, we can use (\ref{eqn_label}) to label all of obtained nodes.
After  collecting the  labeled data pairs $\{\phi_{s},y_s\}$,  we can refine the parameters in the GNN-based classifier through retraining with the aggregation dataset $\cup_t^i D_t$ and obtain $\hat{\theta}^{i+1}$ by solving  the optimization problem in \eqref{IL_equ}.

 Finally we may utilize Algorithm \ref{alg_abb} with the well-trained $\pi_{\theta^*}$ to accelerate the vanilla B\&B algorithm for a new problem instance.   This corresponds to the top right corner of Fig. \ref{BB_framework}. The only difference between GNN-based accelerated B\&B Algorithm \ref{alg_abb} and vanilla B\&B  Algorithm \ref{alg_BB} lies in line 8 of Algorithm \ref{alg_abb}, which uses the trained GNN-based classifier to prune the nodes.

\section{Numerical Results}\label{sec_num}
In this section, we present numerical results to  verify the derived theoretical results and to demonstrate the   efficiency of the proposed B\&B algorithm and its GNN-based accelerated version. We   call CVXPY \cite{diamond2016cvxpy}, which interfaces with MOSEK, to solve  all convex subproblems in the form of \eqref{MER_multi1}. We consider an ISAC BS that is equipped with  $N_r=16$ receive antennas and  $N_t=6$ transmit antennas. The frame length is set as $L=16$.    Unless otherwise specified, the power budget is set as $P_T=30$ dBm and the noise variances are set as $\sigma_C^2=\sigma_s^2=1$.
The following two scenarios with different channel distributions are considered.
\begin{itemize}
  \item \textbf{Scenario 1}: Each entry of the channel vectors $\vh_k$ for all $k\in[K]$ follows the i.i.d. complex Gaussian distribution with zero mean and unit variance.
  \item \textbf{Scenario 2}: The channels between the BS and the users experience Rayleigh fading with a path loss of $32.6+36.7\log_{10}(d)$ dB \cite{wang2022noma}, where $d$ is the distance between the BS and the user in meters. The BS is assumed to be located at  the origin and the users are equally spaced between distances of $50~\text{m}$ and $200~\text{m}$ from the BS.
\end{itemize}

\subsection{Verification of Closed-Form Solutions}
To validate the  correctness of the derived  closed-form solutions in Section III-B, we first examine Scenario 1 and Scenario 2 with a single communication user.
The results are presented in Fig. \ref{closed_form_solution} (a), where the numerical results are obtained by solving the optimization problem in \eqref{beamform_multi} with $K=1$ using Algorithm \ref{alg_BB}. Fig. \ref{closed_form_solution} (a) shows that the derived closed-form solutions match well with the numerical solutions obtained by the proposed Algorithm \ref{alg_BB}.
\begin{figure}[t]
\vspace{-5mm}
  \centering
  \subfloat[Single-user case.]
  {
      \label{closed_form}\includegraphics[width=0.45\textwidth]{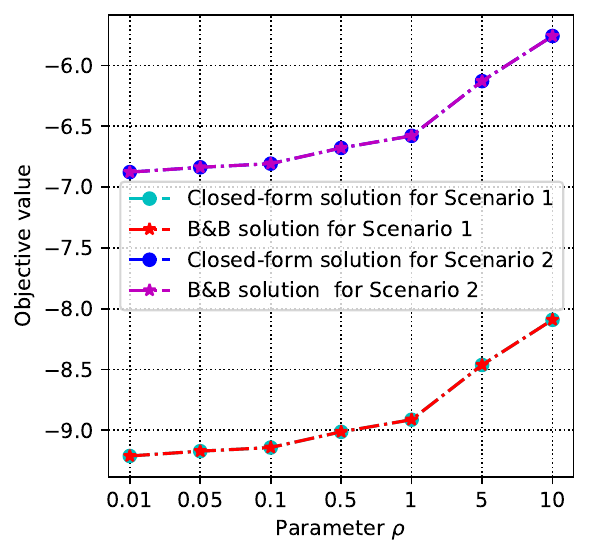}
  }
  \subfloat[Multi-user orthogonal case.]
  {
      \label{closed_form_multi}\includegraphics[width=0.45\textwidth]{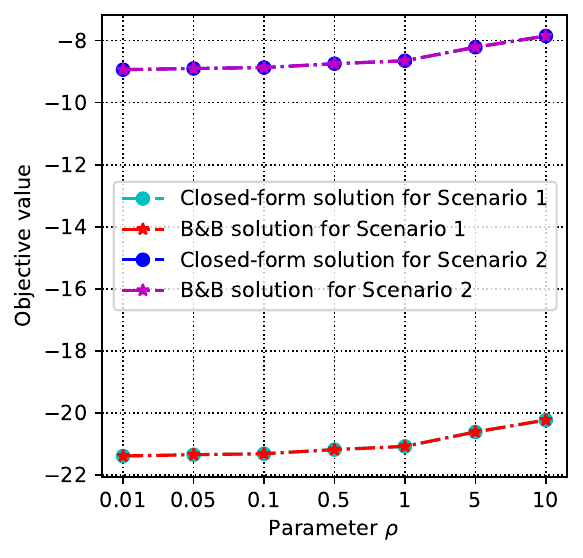}
  }
  \caption{Closed-form and numerical solutions in two special cases.}
  \label{closed_form_solution}
\vspace{-3mm}
\end{figure}

Next, we  consider a multi-user scenario with $K=3$ and orthogonal communication channels. The two curves for each scenario in Fig. \ref{closed_form_solution} (b) are obtained by calling CVXPY to solve convex optimization problem \eqref{multi_user0} and the proposed Algorithm \ref{alg_BB}, respectively.
Fig.  \ref{closed_form_solution} (b) reveals a strong agreement between the solutions obtained from convex optimization and those from the nonconvex problem \eqref{beamform_multi1} using Algorithm \ref{alg_BB}.

\begin{figure}[t]
\vspace{-2mm}
  \centering
  \subfloat[Scenario 1.]
  {
      \label{iteration_eps1}\includegraphics[width=0.45\textwidth]{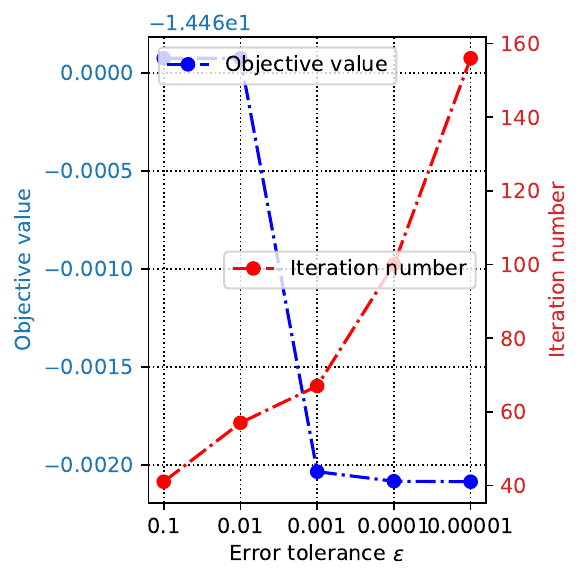}
  }\hspace{-10mm}
  \subfloat[Scenario 2.]
  {
      \label{iteration_eps2}\includegraphics[width=0.45\textwidth]{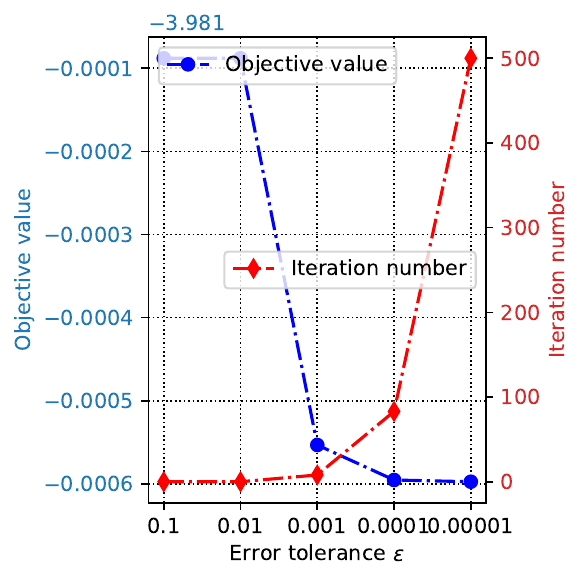}
  }
  \caption{Returned objective values and total iteration numbers
versus the error tolerance $\epsilon$ under two scenarios.}
  \label{iteration_eps}
\vspace{-5mm}
\end{figure}
\subsection{Convergence Behaviors of Proposed Vanilla B\&B Algorithm}
In this subsection,
 we investigate the convergence behaviors of the proposed vanilla B\&B algorithm (i.e., Algorithm \ref{alg_BB}) when applied to solve problem \eqref{beamform_multi1} with $K=3$.
Fig. \ref{iteration_eps} plots the optimal objective  values returned by Algorithm \ref{alg_BB} and the needed total number of iterations versus the error tolerance $\epsilon$.
 On the one hand, it is not surprising to see that the returned objective value  gets better with a decreasing value $\epsilon$.  However, the quality of the returned optimal solution is not sensitive to the value of $\epsilon$ (as the decrease in the returned objective values is marginal).
  On the other hand, the   needed total number of iterations increases   as $\epsilon$ decreases and the increase is drastic when $\epsilon$ is close to zero.
The results in Fig. \ref{iteration_eps} illustrate that the  iteration complexity bound  in \eqref{gmax}
in Theorem \ref{theorem_global} is quite   pessimistic and   the practically needed total iteration number could
be much smaller than   that worst-case bound. It is also clear that   the error tolerance $\epsilon$
provides a tuning knob for achieving various trade-offs between
algorithm's performance and its computational time. In the following simulation, we set $\epsilon=0.001$.

\vspace{-3mm}
\subsection{Comparison with  Existing Local
Optimization Algorithms}
In this subsection, we evaluate the performance of one state-of-the-art algorithm for solving the nonconvex problem \eqref{beamform}, which is called WMMSE-SDR \cite{zhu2023information}. We set $K=3$. In Figs. \ref{sum_rate_user1} and \ref{sum_rate_user2}, we show the trade-off between the communication and sensing performance, i.e., sum rate versus CRB, with
the power budget being set as 20 dBm and 30 dBm, respectively. Total 20 Monte
Carlo realizations  are conducted and the curves in  Figs. \ref{sum_rate_user1} and \ref{sum_rate_user2} are obtained by averaging over them.

\begin{figure}[t]
  \centering
  \subfloat[$P_T=20 $ (dBm).]
  {
      \label{sum_rate_user1_pow20}\includegraphics[width=0.45\textwidth]{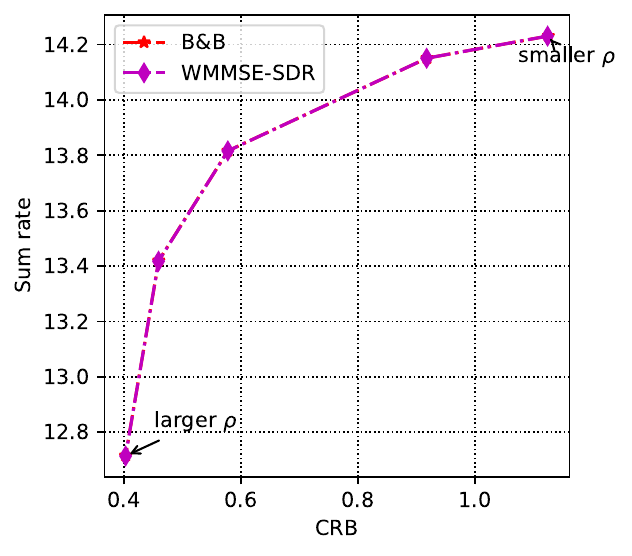}
  }\hspace{-7mm}
  \subfloat[$P_T=30$ (dBm).]
  {
      \label{sum_rate_user1_pow30}\includegraphics[width=0.43\textwidth]{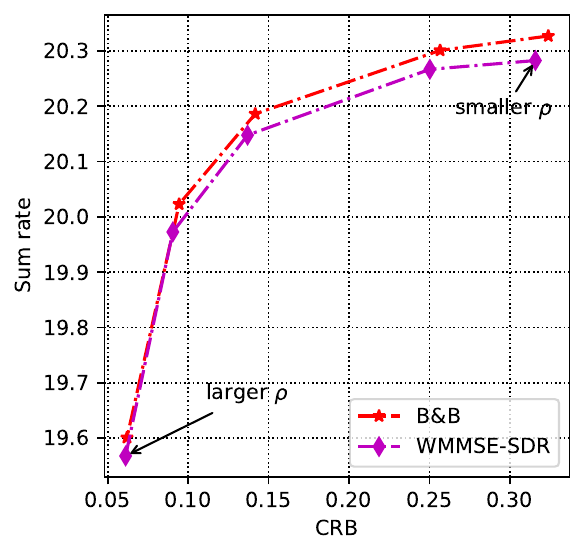}
  }
  \caption{Comparison of proposed vanilla B\&B and WMMSE-SDR algorithms under Scenario 1 with two different  power  budgets.}
 \label{sum_rate_user1}
\vspace{-5mm}
\end{figure}
\begin{figure}[t]
  \centering
  \subfloat[$P_T=20$ (dBm).]
  {
     \label{sum_rate_user2_pow20}\includegraphics[width=0.44\textwidth]{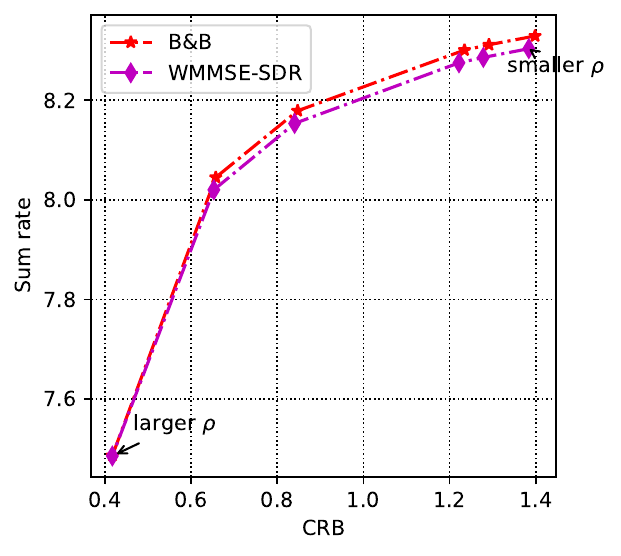}
  }\hspace{-5mm}
  \subfloat[$P_T=30$ (dBm).]
  {
      \label{sum_rate_user2_pow30}\includegraphics[width=0.44\textwidth]{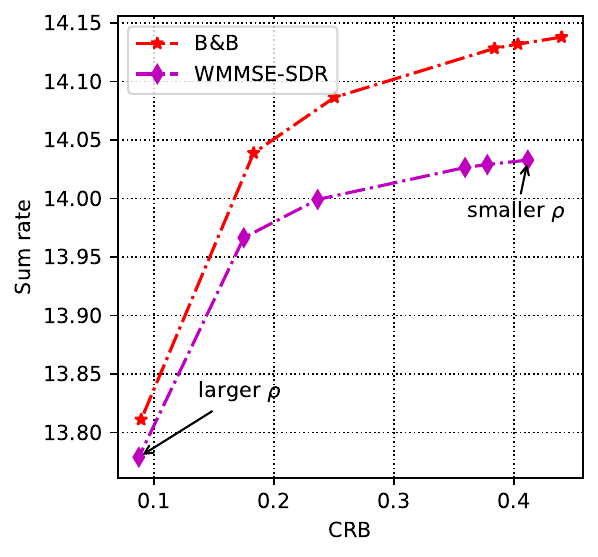}
  }
  \caption{Comparison of proposed vanilla B\&B and WMMSE-SDR algorithms under Scenario 2 with two different power  budgets.}
 \label{sum_rate_user2}
\end{figure}

Figs. \ref{sum_rate_user1} and \ref{sum_rate_user2} show that,   as the power budget increases, the performance gap between WMMSE-SDR algorithm and proposed vanilla B\&B algorithm  becomes larger.
This might be due to the fact that  a large power budget leads to a large feasible set, which might contain  more stationary points. From Figs. \ref{sum_rate_user1}(b) and \ref{sum_rate_user2}(b), we can also observe that proposed vanilla B\&B algorithm performs better than WMMSE-SDR algorithm when   the parameter $\rho$ is smaller.   This is because the parameter $\rho$ affects the nonconvexity of the objective function of problem (\ref{beamform}), i.e., the smaller $\rho$ is, the higher the nonconvexity of the objective function of problem (\ref{beamform}). All of these results show  the performance gain of the proposed
global optimization algorithm over that of the local optimization algorithm.

\subsection{Performance of GNN-based Accelerated B\&B Algorithm}
In this subsection, we demonstrate the efficacy of the proposed GNN-based accelerated B\&B algorithm.
The loss function is
selected to be the binary cross-entropy loss. The Adam algorithm is employed for training the GNN over 20 epochs, with the batch size and the step size being 128 and 0.001, respectively.

To address the class imbalance problem, where the number of relevant nodes is typically much smaller than the number of irrelevant nodes in the training set, we assign a higher weight to the ``positive"  training pairs. Empirically an early pruning of nodes in  the B\&B tree is   not preferred due to   an increased risk of pruning the nodes containing the global solution.
Hence, following \cite{shrestha2023optimal}, we weight each term $\mathcal{L}(\pi_{\theta}(\phi_s),y_s)$  using
\begin{align*}
\left\{
  \begin{array}{ll}
    \frac{1}{d}, & \hbox{if $y_s=0$;} \\[1mm]
    \frac{1+q}{d}, & \hbox{if $y_s=1$,}
  \end{array}
\right.
\end{align*}
 where $d$ is the depth of node $s$ and $q\in \mathbb{R}$ offsets the imbalance ratio. In all experiments, we set $q = 11$.
\begin{figure}[t]
\begin{center}
\includegraphics[width=0.6\linewidth]{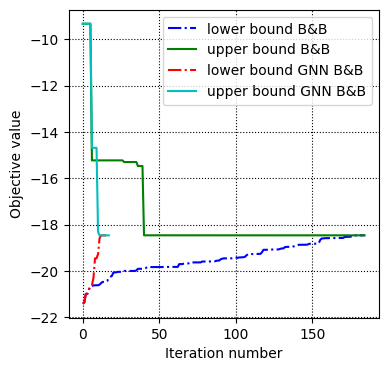}
\caption{Comparison of the lower and upper bounds in vanilla B\&B and GNN-based accelerated B\&B algorithms.}
\label{ml_BB}
\end{center}
\end{figure}
\begin{table}[t]
  \centering
\vspace{-2mm}
   \caption{Performance of GNN-based accelerated B\&B algorithm.
}\label{table1}
   \renewcommand\arraystretch{2}
 \begin{tabular}{|c|c|c|c|c|}
   \hline
   \multirow{2}{*} {$(K,N_t)$}&
   \multicolumn{2}{c|}{Scenario 1} &\multicolumn{2}{c|}{Scenario 2} \\
   \cline{2-5} &  Ogap & speedup & Ogap & speedup\\
   \hline
   $(3,6)$  & 0.01&5.68 &0.004 &6.87 \\
  \hline
   $(3,8)$  & 0.04&3.28 &0.0001 &4.20 \\
  \hline
   $(3,10)$  & 0.05&3.33 &0.001 & 5.24 \\
  \hline
   \end{tabular}
\end{table}

  We define two metrics to compare the vanilla B\&B and GNN-based accelerated B\&B algorithms. The first one is
 the optimality gap (Ogap) as follows:
\begin{align*}
\textmd{Ogap} := \frac{\hat{U}_{\text{GNN}}-U^*}{|U^*|} \times 100\%,
\end{align*}
where $U^*$ is the optimal objective value  returned by the vanilla B\&B algorithm and $\hat{U}_{\text{GNN}}$ is the   optimal objective value returned by GNN-based accelerated B\&B algorithm. The second metric is the running time speedup:
\begin{align*}
\textmd{speedup} := \frac{\textmd{Running time of vanilla B\&B (seconds)}}{\textmd{Running time of  accelerated B\&B (seconds)}}.
\end{align*}

Fig. \ref{ml_BB} plots the convergence behaviors of lower and upper bounds of proposed vanilla and
 GNN-based accelerated B\&B algorithms when $K=3$ under Scenario 1. It can be seen from Fig. \ref{ml_BB} that both lower and upper bounds of the vanilla B\&B algorithm converge to the optimal objective value, which is consistent with the global optimality guarantee in Theorem \ref{theorem_global}.  More importantly, GNN-based accelerated B\&B algorithm converges significantly faster than the vanilla B\&B algorithm, which shows that the learned GNN-based classifier indeed safely and quickly prunes many irrelevant nodes.

 Table \ref{table1} showcases the performance of the proposed   GNN-based accelerated  B\&B algorithm for various problem sizes. The results are averaged over 20 random test instances.  It can be observed from Table \ref{table1} that the GNN-based accelerated B\&B algorithm can achieve a 3-7 times speedup over the vanilla B\&B algorithm without sacrificing the solution quality. In particular, the (average) optimality gap between the two algorithms are less than 0.05\%.

%

\section{Conclusion}\label{sec_con}
In this paper, we have proposed an optimization model for transmit beamforming design in MIMO ISAC systems. More specifically, we have formulated the beamforming design problem as the maximization of a weighted sum of the sum rate of all communication users and the CRB of the extended target subject to the power budget constraint at the BS. The focus of this paper is to obtain   the global solution for the formulated nonconvex problem. We have firstly derived an optimal closed-form solution to the
formulated problem in two special cases. Furthermore, we have proposed vanilla B\&B algorithm and GNN-based accelerated B\&B algorithm for solving the formulated problem in the general multi-user scenario. While the vanilla B\&B algorithm is guaranteed to find the global solution, its GNN-based accelerated version is able to significantly improve its computational efficiency by embedding a learned pruning policy   in the vanilla B\&B algorithm. Numerical results have shown the correctness of the derived theoretical results as well as the efficiency of proposed B\&B algorithms.
In the future, we   plan to consider the beamformer design for   ISAC systems with multiple BSs.

		\vspace{-0.0cm}
		\appendices

\section{Proof of Proposition \ref{theorem_rank} }\label{proof_rank}
Since
 the objective function of problem \eqref{beamform_multi1} only depends on $\mR_X$ and $\{\Gamma_k\}_{k=1}^K$ and $ \tilde{\mR}_X= \bar{\mR}_X, \tilde{\Gamma}_k = \bar{\Gamma}_k$ for all $k\in[K]$, we only need to show that the constructed solutions $\bar{\mR}_X, \{\bar{\Gamma}_k\}_{k=1}^K, \{\bar{\mW}_k\}_{k=1}^K$ in (\ref{rrx}) is feasible to problem \eqref{beamform_multi1}.

 First,  from $\mQ_k=\vh_k\vh_k^H$ for all $k\in[K]$, one can obtain
    \begin{align*}
 \tr(\mQ_k\bar{\mW}_k\mQ_k\mW_k^H)&=\tr(\vh_k^H\bar{\mW}_k\vh_k^H\vh_k\bar{\mW}_k^H\vh_k)\\
 &=(\vh_k^H\bar{\mW}_k\vh_k)^2=(\tr(\mQ_k\bar{\mW}_k))^2.
 \end{align*}
 Using it and substituting $\{\tilde{\Gamma}_k\}_{k=1}^K,\{\tilde{\mW}_k\}_{k=1}^K,\tilde{\mR}_X$ into \eqref{beamform_multi_cons13}, we can get
 \begin{align*}
 &(1-\tilde{\Gamma}_k)\tr(\mQ_k\tilde{\mW}_k)-\tr(\mQ_k\tilde{\mR}_X)\\
 =&~(1-\bar{\Gamma}_k)\tr(\mQ_k\bar{\mW}_k)-\tr(\mQ_k\bar{\mR}_X)\geq \tilde{\Gamma}_k\sigma_C^2,
  \end{align*}
which shows that constraint in \eqref{beamform_multi_cons13} holds with $\{\tilde{\Gamma}_k\}_{k=1}^K,\{\tilde{\mW}_k\}_{k=1}^K,\tilde{\mR}_X$.

Second, we show that $\bar{\mW}_k-\tilde{\mW}_k\succeq 0$ for all $k\in[K]$. For any vector $\vv$, it holds that
\begin{align*}
\vv^H(\bar{\mW}_k-\tilde{\mW}_k)\vv=\vv^H\bar{\mW}_k\vv-(\vh_k^H\bar{\mW}_k\vh_k)^{-1}|\vv^H\bar{\mW}_k\vh_k^H|^2.
\end{align*}
Using the Cauchy-Schwarz inequality, we have
\begin{align*}
(\vh_k^H\bar{\mW}_k\vh_k)(\vv^H\bar{\mW}_k\vv)\geq |\vv^H\bar{\mW}_k\vh_k|^2,
\end{align*}
{which implies} $\bar{\mW}_k-\tilde{\mW}_k\succeq 0$. As such, we get
\begin{align*}
 \tilde{\mR}_X-\sum_{k\in[K]}\tilde{\mW}_k =  \bar{\mR}_X-\sum_{k\in[K]}\bar{\mW}_k+\sum_{k\in[K]}(\bar{\mW}_k-\tilde{\mW}_k) \succeq 0,
\end{align*}
namely, constraint  \eqref{beamform_multi_cons12}  holds with $\tilde{\mR}_X, \{\tilde{\Gamma}_k\}_{k=1}^K, \{\tilde{\mW}_k\}_{k=1}^K$.
\section{Proof of Theorem \ref{theorem_closed}}\label{proof_theorem2}
First, given the optimal $\mR_X$ and $\Gamma_1$ of problem (\ref{single_user0}), we can construct its optimal $\mW_1$ as follows:
\begin{align}\label{W_1}
\mW_1 = (\vu_1^H\mR_X\vu_1)\vu_1\vu_1^H,
\end{align}
where $\vu_1 = {\vh_1}/{\|\vh_1\|}$. The reason is as follows.
By \eqref{W_1}, we have
\begin{align}\label{qw1}
\nonumber\tr(\mQ_1\mW_1)&=(\vu_1^H\mR_X\vu_1)\tr(\vh_1\vh_1^H\vu_1\vu_1^H)\\
&=\tr(\mQ_1\mR_X)
\end{align}
and thus $\tr\left(\mQ_1(\mR_X-\mW_1)\right)=0$.

Note that $\mR_X=\mW_1+\mW_2$.  Substituting  this into \eqref{qw1}, we have
\begin{align*}
\tr\left(\mQ_1\mW_2\right) = \vh_1^H\mW_2\vh_1 =0.
\end{align*}
Then it is simple to check that $\mW_2\vu_1 =0$  and  $\vu_1$ is an eigenvector of the
optimal $\mR_X$. Without loss of generality, let the eigenvalue decomposition of the
optimal $\mR_X$ be
\begin{align}\label{eig_RX}
\mR_X = \sum_{i\in[N_t]}\lambda_{i}\vu_i\vu_i^H,
\end{align}
where $\lambda_{i}$ is  the $i$-th eigenvalue and  $\vu_i$ is the corresponding eigenvector.
%
Substituting  \eqref{W_1} and \eqref{eig_RX} into problem \eqref{single_user0}, we obtain
\begin{subequations}\label{single_user}
\begin{align}
\label{single_user1}\min_{\substack{\{\lambda_{i}\}_{i=1}^{N_t},\Gamma_1}}&~ -\log\left(1+\Gamma_1\right)+ \rho \sum_{i\in[N_t]}\lambda_{i}^{-1} \\
 \textmd{s.t.} ~~~& ~ ~\lambda_{1}\|\vh_1\|^2\geq \Gamma_1\sigma_C^2,\\
 &~\sum_{i\in[N_t]}\lambda_{i}\leq P_T, ~ \lambda_{i}> 0,~ i \in [N_t].
\end{align}
\end{subequations}
Obviously, problem \eqref{single_user} is a convex problem.  Next, we   derive the optimal solution of problem \eqref{single_user}.

The Lagrangian function of problem \eqref{single_user} is
 \begin{align}
 \nonumber \mathcal{L}=&-\log\left(1+\Gamma_1\right)+ \rho \sum_{i\in[N_t]}\lambda_{i}^{-1}+\omega\left(-\lambda_{1}+\frac{\Gamma_1\sigma_C^2}{\|\vh_1\|^2}\right)\\
 &+\mu\Big(\sum_{i\in[N_t]}\lambda_{i}-P_T\Big)-\sum_{i\in[N_t]}\eta_i\lambda_{i},
 \end{align}
 where $\omega$, $\mu$, and $\left\{\eta_i\right\}_{i=1}^{N_t}$ are the Lagrange multipliers associated with the inequality constraints in problem \eqref{single_user}. Then, the
KKT conditions of problem (\ref{single_user}) can be given as follows:
\begin{subequations}\label{KKT_trace}
\begin{align}
\label{KKT_trace1}\frac{\partial\mathcal{L}}{\partial \lambda_{1}} &= -\rho\lambda_{1}^{-2}-\omega+\mu-\eta_1 = 0,\\
\label{KKT_trace2}\frac{\partial\mathcal{L}}{\partial \lambda_{i}} &= -\rho\lambda_{i}^{-2}+\mu-\eta_i = 0, ~i=2,3,\ldots,N_t,\\
\label{KKT_trace3}\frac{\partial\mathcal{L}}{\partial \Gamma_1} &= -\frac{1}{1+\Gamma_1}+\frac{\omega\sigma_C^2}{\|\vh_1\|^2} = 0,\\
\label{KKT_trace4}\omega&\left(-\lambda_{1}+\frac{\Gamma_1\sigma_C^2}{\|\vh_1\|^2}\right)=0, ~\omega\geq0,~
\lambda_{1}\geq\frac{\Gamma_1\sigma_C^2}{\|\vh_1\|^2},\\
\mu&\left(\sum_{i=1}^{N_t}\lambda_{i}-P_T\right)=0, ~\mu \geq 0,~ \sum_{i=1}^{N_t}\lambda_{i}\leq P_T,\\
\eta_i& \lambda_{i} = 0, ~\eta_i\geq 0, ~\lambda_{i}> 0,~ i\in[N_t].
\end{align}
\end{subequations}
Since $\lambda_{i}> 0$ for all $i\in[N_t]$, we have $\eta_i=0$ for all $i\in[N_t]$. Hence, \eqref{KKT_trace1}--\eqref{KKT_trace3}
can be  simplified into
\begin{subequations}\label{KKT_trace_new}
\begin{align}
\label{KKT_trace_new1}\rho\lambda_{1}^{-2} &= \mu-\omega,\\
\label{KKT_trace_new2}\rho\lambda_{i}^{-2} &=\mu,\\
\label{KKT_trace_new3} \frac{1}{1+\Gamma_1} &= \frac{\omega\sigma_C^2}{\|\vh_1\|^2}.
\end{align}
\end{subequations}
From \eqref{KKT_trace_new2}, we must have $\mu>0$, which implies the   inequality power budget constraint must hold with equality, i.e., $ \sum_{i=1}^{N_t}\lambda_{i}= P_T$. In addition,  it follows from
\eqref{KKT_trace_new3} that $\omega>0$.   Furthermore, it follows from the positiveness of $\omega$ and the first equation in (\ref{KKT_trace4}) that $\lambda_{1}={\Gamma_1\sigma_C^2}/{\|\vh_1\|^2}$,  which, together with the equality power budget constraint, gives
$\sum_{i=2}^{N_t}\lambda_{i}=P_T-{\Gamma_1\sigma_C^2}/{\|\vh_1\|^2}$. Using \eqref{KKT_trace_new2}, we obtain
\begin{align}\label{lambda_ii}
\lambda_{i} = \frac{P_T\|\vh_1\|^2-\Gamma_1\sigma_C^2}{\|\vh_1\|^2(N_t-1)}, ~i=2,3,\ldots,N_t.
\end{align}
From \eqref{KKT_trace_new1}--\eqref{KKT_trace_new2},   we can further obtain
\begin{align}
\omega = \rho(\lambda_{i}^{-2}-\lambda_{1}^{-2}).
\end{align}
Substituting \eqref{lambda_ii} into the above equation and using \eqref{KKT_trace_new3}, we   get \eqref{gamma_equation}.

%
\section{Proof of Lemma \ref{multi_lemma}}\label{proof_lemma_mut}
Given the optimal $\mR_X$ and $\{\Gamma_k\}_{k=1}^K,$   we can construct the optimal $\mW_k$ as follows:
\begin{align}\label{W_k}
\mW_k = (\vu_k^H\mR_X\vu_k)\vu_k\vu_k^H,
\end{align}
where $\vu_k = {\vh_k}/{\|\vh_k\|}$ for all $k\in[K]$.  This is because
  \begin{align*}
  &\tr(\mQ_k(\mR_X-\mW_k))\\
  =&\tr(\mQ_k\mR_X)-(\vu_k^H\mR_X\vu_k)\tr(\mQ_k\vu_k\vu_k^H)=0, ~k\in[K].
  \end{align*}
  Now we study the optimal solution of $\mR_X.$
Substituting $\mR_X=\sum_{k=1}^{K+1}\mW_k$ into \eqref{inter_Qk}, we get
\begin{align}\label{mqw}
\sum_{j\neq k}^{K}\tr(\mQ_k\mW_j)+\tr(\mQ_k\mW_{K+1})=0.
\end{align}
Since  $\vh_1,\vh_2,\ldots,\vh_K$ are orthogonal to each other,   it follows that
\begin{align*}
\tr(\mQ_k\mW_j) &= \vh_k^H\mW_j\vh_k=0,~ j\neq k, ~j,~k\in[K].
\end{align*}
Substituting the above equation into \eqref{mqw}, we have
\begin{align*}
\tr(\mQ_k\mW_{K+1}) &=\vh_k^H\mW_{K+1}\vh_k=0,~k\in[K].
\end{align*}
  Then it is simple to obtain
\begin{align*}
\mR_X\vu_k=\mW_k\vu_k=(\vu_k^H\mR_X\vu_k)\vu_k,~k\in[K],
\end{align*}
which   implies that $\lambda_{k}=\vu_k^H\mR_X\vu_k$ is one of the eigenvalue  of the optimal $\mR_X$ and  $\vu_k$ is   the corresponding  eigenvector for all $k\in[K]$.
\section{Proof of Lemma \ref{lemma_con}}\label{proof_lemma_con}
  From line 17 in Algorithm \ref{alg_BB}, we have $U^t\leq \hat{U}^t$,  which gives rise to
\begin{align}\label{eqn:UL}
U^t-L^t&\leq \hat{U}^t-L^t= \sum_{k\in [K]}\log\left(1+\frac{\Gamma_k^t-\hat{\Gamma}_k^t}{1+\hat{\Gamma}_k^t}\right),
\end{align}
where the last equality is due to \eqref{eqn:upper_obj} and
 the fact
\begin{align*}
 L^t =-\sum_{k\in[K]}\log(1+{\Gamma}_k^t)+ \rho \tr((\mR_X^t)^{-1}).
\end{align*}
As we  use \eqref{eqn:k}  to obtain the  index   $k^*$,   we further have
 \begin{align*}
 U^t-L^t\leq K\log\left(1+\frac{\Gamma_{k^*}^t-\hat{\Gamma}_{k^*}^t}{1+\hat{\Gamma}_{k^*}^t}\right).
\end{align*}
Since $0\leq \ell_{k^*}^t\leq\hat{\Gamma}_{k^*}^t\leq u_{k^*}^t$ and $\Gamma_{k^*}^t\leq u_{k^*}^t$, it follows that
 \begin{align}\label{eqn:UL2}
\nonumber U^t-L^t&\leq K\log\left(1+u_{k^*}^t-\ell_{k^*}^t\right)\\
&\leq K\log\left(1+2\delta_{\epsilon}\right).
\end{align}
This, together with the definition of $\delta_{\epsilon}$ in \eqref{eqn:delta}, shows the desirable result \eqref{eqn:term}.

\section{Proof of Theorem \ref{theorem_global}}\label{proof_theorem_global}
As shown in Lemma \ref{lemma_con}, when \eqref{lemma_cond} is satisfied, the proposed Algorithm \ref{alg_BB} will terminate and return an $\epsilon$-optimal solution. In this part, we show  that Algorithm \ref{alg_BB} will terminate within   $T_{\epsilon}$ iterations based on a contradiction principle,  where $T_{\epsilon}$ is given in \eqref{gmax}.

Suppose that the algorithm does not terminate within $T_{\epsilon}$ iterations.  From Lemma \ref{lemma_con}, we know that condition \eqref{lemma_cond} does not hold, i.e., $u_{k^*}^t-\ell_{k^*}^t\geq 2\delta_{\epsilon}$ for all $t\in[T_{\epsilon}].$   According to line 9 of Algorithm \ref{alg_BB}, the width of the two sub-intervals $[\ell_{k^*}^t,z_{k^*}^t]$ and $[z_{k^*}^t,u_{k^*}^t]$   after the partition is greater than $\delta_{\epsilon}$. Then, for each subset $\mathcal{Q}^t:=\prod_{k=1}^K[\ell_k^t,u_k^t]$   obtained from the partition of  the original set $\mathcal{Q}^0$,  there   must hold $u_{k}^t-\ell_{k}^t\geq \delta_{\epsilon}$ for all $k\in[K]$.
As a result,  the volume of each subset $\mathcal{Q}^t$ is not less than $\delta_{\epsilon}^{K}$ and the total volume of all $T_{\epsilon}$ subsets is not less than $T_{\epsilon}\delta_{\epsilon}^{K}$. In addition,  the volume of $\mathcal{Q}^0$ is less than $\Gamma_{\max}^K$. By the definition of $T_{\epsilon}$, we get $$T_{\epsilon}\delta_{\epsilon}^{K}>\Gamma_{\max}^K,$$ which implies that the total volume of all $T_{\epsilon}$ subsets is greater than that of the original set $\mathcal{Q}^0$. This is a contradiction. Therefore, the algorithm will terminate within at most $T_{\epsilon}$ iterations.
\section{Input Feature Design}\label{feature_design}
In this appendix, we specify the detailed
 features $\{\vx_n\in \mathbb{R}^1\mid n\in [N_t]\}$, $\{\vx_{N_t+k}\in\mathbb{R}^{13}\mid k\in [K]\}$, and $\{\ve_{n,N_t+k}\in\mathbb{R}^{4}\mid n\in[N_t], k\in[K]\}$ used in the GNN in Section V-B.
  \begin{itemize}
    \item The feature at the antenna node: we compute   all eigenvalues of matrix $\mR_X^t$ and let $\vx_n$ represent the $n$-th  eigenvalue of matrix $\mR_X^t;$
    \item The feature at the user node: the length of the feature vector for each user node is 13 and each element is given as follows:
 \begin{align*}
&\vx_{N_t+k}(1) = \ell^t, ~\vx_{N_t+k}(2) =u^t, ~\vx_{N_t+k}(3) =\hat{\Gamma}^t,\\
 &\vx_{N_t+k}(4) =\Gamma^t, ~\vx_{N_t+k}(5) =U^t,~ \vx_{N_t+k}(6) =L^t,\\
 &\vx_{N_t+k}(7) = \mathbb{I}(\hat{U}^{t}-U^t\leq \epsilon), ~\vx_{N_t+k}(8) = d, \\
 & \vx_{N_t+k}(9) = \tr(\mQ_k\mW_k^t),\\
  &\vx_{N_t+k}(10)  = \tr(\mQ_k(\mR_X^t-\mW_k^t)),\\
 &  \vx_{N_t+k}(11) =\bar{L}^t,  ~ \vx_{N_t+k}(12)= \hat{U}^t, ~\vx_{N_t+k}(13)= \bar{\Gamma}^t,
 \end{align*}
where $\mathbb{I}(\cdot)$ denotes the indicator function;
    \item The feature at the edge: the length of the feature
vector for each edge is 4 and each element is given
as follows:
 \begin{align*}
\ve_{n,N_t+k}(1) &= \textmd{Re}(\mH_{n,jk}),~ \ve_{n,N_t+k}(2) = \textmd{Im}(\mH_{n,k}), \\
\ve_{n,N_t+k}(3) &= |\mH_{n,k}|,
 \end{align*}
where  $|\cdot|$ denotes the modulus of a complex value.
 We also compute the eigenvalue of matrix $\mW_k^t$ for all $k\in[K]$ and let $\ve_{n,N_t+k}(4)$ represent the $n$-th  eigenvalue of matrix $\mW_k^t$.
  \end{itemize}

%


\end{document}